\documentclass[11pt,reqno]{amsart}
\usepackage[pagewise]{lineno}
\usepackage{pgfplots}
\pgfplotsset{compat=1.18}
\usetikzlibrary{patterns}

\usepackage{amsmath}
\usepackage{mathrsfs}
\usepackage{amsfonts}
\usepackage{amssymb}
\usepackage[pagewise]{lineno}
\usepackage{url}
\theoremstyle{plain}
\newtheorem{athm}{Theorem}

\newtheorem{theorem}{Theorem}[section]
\newtheorem{proposition}[theorem]{Proposition}
\newtheorem{corollary}[theorem]{Corollary}
\newtheorem{lemma}[theorem]{Lemma}

\theoremstyle{definition}

\newtheorem{example}{Example}[section]
\newtheorem{definition}[theorem]{Definition}

\newtheorem{remark}[theorem]{Remark}

\DeclareMathOperator{\esssup}{ess \ sup}

\DeclareMathOperator{\ho}{H}
\DeclareMathOperator{\e}{e}

\DeclareMathOperator{\Ker}{Ker}
\DeclareMathOperator{\im}{Im}

\input{tcilatex}

\begin{document}

\begin{abstract}
This paper establishes limit theorems and quantitative statistical stability for a class of piecewise partially hyperbolic maps that are not necessarily continuous nor locally invertible. By employing a flexible functional-analytic framework that bypasses the classical requirement of compact embeddings between Banach spaces, we obtain explicit rates of convergence for the variation of equilibrium states under perturbations. Furthermore, we prove the exponential decay of correlations and the Central Limit Theorem for H\"older observables. A key feature of our approach is its applicability to systems where traditional spectral gap techniques fail due to the presence of singularities and the lack of invertibility. We provide several examples illustrating the scope of our results, including partially hyperbolic attractors over horseshoes, non-invertible dynamics semi-conjugated to Manneville--Pomeau maps, and fat solenoidal attractors.
\end{abstract}

\title[Limit Theorems and Stability for P. Partially Hyperbolic E. States]{Limit Theorems and Quantitative Statistical Stability for the Equilibrium States of Piecewise Partially Hyperbolic Maps}
\author[Rafael A. Bilbao]{Rafael A. Bilbao}
\author[Rafael Lucena]{Rafael Lucena}

\date{\today }

\address[Rafael A. Bilbao]{Universidad Pedag\'ogica y Tecnol\'ogica de Colombia, Avenida Central del Norte 39-115, Sede Central Tunja, Boyac\'a, 150003, Colombia.}
\email{rafael.alvarez@uptc.edu.co}

\address[Rafael Lucena]{Universidade Federal de Alagoas, Instituto de Matemática - UFAL, Av. Lourival Melo Mota, S/N
	Tabuleiro dos Martins, Maceio - AL, 57072-900, Brasil}
\email{rafael.lucena@im.ufal.br}
\urladdr{www.im.ufal.br/professor/rafaellucena}


\maketitle

\section{Introduction}

The qualitative theory of differential equations and smooth dynamical systems is fundamentally concerned with the long-term behaviour of orbits and the robustness of their statistical descriptions. A central theme in this arena is statistical stability, the study of how invariant measures, particularly equilibrium states, respond to perturbations of the underlying evolution law. Since the foundational works of Alves, Bonatti, and Viana \cite{ABV}, the qualitative continuity of these measures has been established for various hyperbolic and partially hyperbolic systems. However, obtaining quantitative statistical stability, explicit rates of convergence for the variation of measures, remains a challenging frontier, especially for systems that deviate from standard smoothness and invertibility assumptions.

A landmark in this direction was the framework developed by Keller and Liverani \cite{KL}, which established that the stability of the leading eigenvector of a transfer operator could be controlled by the closeness of the operators in a suitable "weak" norm, provided they satisfy a uniform Lasota-Yorke inequality. This approach proved to be exceptionally powerful for hyperbolic systems and has been the gold standard for deriving rates of convergence. Moreover, the perturbative framework, pioneered by Keller and Liverani, typically operates on a pair of Banach spaces $(B_s, B_w)$ under the crucial assumption that the strong space $B_s$ is compactly embedded in the weak space $B_w$. This compactness serves as a fundamental pillar of the theory, as it is essential for invoking the Hennion--Ionescu-Tulcea \cite{He, IT} theorem to guarantee the quasi-compactness of the transfer operator and the subsequent stability of its isolated eigenvalues. Nevertheless, in the presence of complex singularity sets or non-smooth dynamics, establishing such compact embeddings often constitutes a substantial technical barrier.

Regarding smooth, non-uniformly hyperbolic dynamics, Baladi and Viana \cite{BaVi} established the stochastic stability of unimodal maps. Shifting to systems with singularities, Demers and Liverani \cite{DL} proved statistical stability for two-dimensional piecewise hyperbolic maps. A fundamental ingredient in their approach was the construction of Banach spaces of distributions in which the strong space is compactly embedded into the weak space.

In the present work, we investigate skew-product maps of the form $F(x,y) = (f(x), G(x,y))$, characterized by a non-uniformly expanding base dynamics $f$ and fiber maps $G(x, \cdot)$ that exhibit contraction on almost every fiber, even in the presence of discontinuity sets parallel to the contraction direction. In the qualitative theory of differential equations, such discrete-time systems naturally emerge as Poincaré return maps for flows subject to singularities, holonomic constraints, or impulsive impacts. For these systems, we establish the quantitative statistical stability, the central limit theorem, and the exponential decay of correlations for their equilibrium states, considering a broad class of observables.

The methodology employed herein is rooted in the novel approach to the thermodynamic formalism of discontinuous maps introduced in \cite{RR}, which builds upon the foundational techniques of \cite{RRR, GLu}. A central technical innovation of this functional framework, further developed and utilized in the present work, lies in its flexibility. Notably, the framework bypasses the standard requirement for a compact embedding of the strong space into the weak space. Furthermore, the vector spaces upon which the transfer operator acts are not constrained to be Banach spaces, thereby facilitating a more versatile and adaptable analytical treatment of non-smooth dynamics.

The present paper provides a significant leap forward, utilizing the flexible spectral gap results of \cite{RR} to provide a complete statistical description of these discontinuous maps. Our contribution pushes the boundaries of the existing literature, specifically extending the framework of \cite{RRRSTAB} in several fundamental directions:

\begin{itemize}
    \item \textbf{Quantitative Statistical Stability:} We establish that the equilibrium state $\mu_\delta$ satisfies a quantitative modulus of continuity of the form $D R(\delta)^\zeta \log \delta$ with respect to the perturbation parameter $\delta$;
    \item \textbf{Non-Invertible Dynamics:} Our framework encompasses maps that are not necessarily invertible nor locally invertible, addressing a significant gap in the qualitative theory of non-smooth dynamical systems;

\item \textbf{Generalized Potentials and Dynamics:} We extend the statistical analysis to a significantly broader class of potentials $\overline{\mathscr{P}}_\Sigma$ and dynamical systems than those treated in \cite{RRRSTAB}. Notably, our results encompass potentials with explicit dependence on the fiber direction, thereby eliminating the restrictive requirement of fiber-constancy;
    
    \item \textbf{The Central Limit Theorem (CLT):} Most notably, this paper demonstrates for the first time how this flexible operator approach can be applied to prove the CLT, showing that fluctuations of Birkhoff averages converge to a normal distribution.
\end{itemize}


\textbf{Acknowledgment} This work was partially supported by CNPq (Brazil) Grants 446515/2024-8 and CNPq (Brazil) Grants 420353/2025-9.

  \subsection*{Organization of the paper}
The remainder of this article is organized as follows. In Section~\ref{kjdfkjdsfkj}, we formalize the setting for piecewise partially hyperbolic maps and discuss several illustrative examples. Section~\ref{seccc} is devoted to preliminary concepts, specifically the construction of the anisotropic spaces $\mathcal{H}^\zeta_m$, $\mathbf{L}^\infty_m$, and $\mathbf{S}^\infty$, along with the required metric tools. In Section~\ref{from}, we develop the functional-analytic framework for the transfer operator. Section~\ref{limit} focuses on the statistical properties of the system, where we prove the exponential decay of correlations and the Central Limit Theorem; precisely, we provide the proofs of Theorems~\ref{çljghhjçh}, \ref{athmc}, and \ref{central}. In Section~\ref{loeritu}, we introduce the class of perturbations under consideration, describe the functional-analytic perturbative framework, and establish our main stability results, concluding with the proof of Theorem~\ref{d}. Finally, in Section~\ref{ex}, we extend our analysis to equilibrium states associated with potentials that are not constant along fibres, providing the proofs of Theorems~\ref{e}, \ref{f}, \ref{g}, and \ref{h}.

\section{Settings and Examples}\label{kjdfkjdsfkj}

Throughout this article, we consider skew--product maps
$F : M \times K \longrightarrow M \times K$ of the form $F=(f,G)$, where
$f : M \longrightarrow M$ and $G : M \times K \longrightarrow K$
are measurable maps specified below.
We fix a probability measure $m$ on the compact and connected metric space $(M,d_1)$ satisfying the hypotheses listed in (P1),
and an arbitrary probability measure $m_2$ on the Borel $\sigma$--algebra of the
compact metric space $(K,d_2)$.
For convenience, we set $\Sigma := M \times K$.

\subsection{Hypothesis on the basis map $f$.}

Suppose that $f:M \longrightarrow M$ is a local homeomorphism and assume that there is a continuous function $L:M\longrightarrow\mathbb{R}$, s.t. for every $x \in M$ there exists a neighborhood $U_x$, of $x$, so that $f_x:=f|_{U_x}: U_x \longrightarrow f(U_x)$ is invertible and 
\begin{equation}\label{iirotyirty}
	d_1(f_x^{-1}(y), f_x^{-1}(z)) \leq L(x)d_1(y, z), \ \ \forall y,z \in f(U_x).
\end{equation}In particular, $\#f^{-1}(x)$ is constant for all $x \in M$. We set $\deg(f):=\#f^{-1}(x)$, the degree of $f$.

Suppose that there is an open region $\mathcal{A} \subset M$ and constants $\sigma >1$ and $L\geq 1$ such that  
\begin{enumerate}
	\item[(f1)] $L(x) \leq L$ for every $x \in \mathcal{A}$ and $L(x) < \sigma ^{-1}$ for every $x \in \mathcal{A}^c$. Moreover, $L$ is close enough to $1$ (a precise bound for $L$ is given in equation (\ref{kdljfhkdjfkasd}));
	\item[(f2)] There exists a finite covering $\mathcal{U}$ of $M$, by open domains of injectivity for $f$, such that $\mathcal{A}$ can be covered by $q<\deg(f)$ of these domains.
\end{enumerate}The first condition means that we allow expanding and contracting behavior to coexist in M: $f$ is uniformly expanding outside $\mathcal{A}$ and not too contracting inside $\mathcal{A}$. In the case that $\mathcal{A}$ is empty then $f$ is uniformly expanding. The second one requires that every point has at least one preimage in the expanding region.

We also assume that $f$ satisfies the full-branch condition (P2). 

\begin{enumerate}
    \item[(P2)] There exists a partition of $M$ into measurable sets $\{P_1, \dots, P_{\deg(f)}\}$, disjoint up to $\nu$-null sets, such that each restriction $f_i = f|_{P_i}$ is a bijection onto $M$ and 
    \begin{equation*}
        \nu\left( M \setminus \bigcup_{i=1}^{\deg(f)} P_i \right) = 0.
    \end{equation*}
\end{enumerate}

\begin{definition}\label{ejhrkjhe}
	For a given $0 <\zeta \leq 1$, denote by $H_\zeta$ the set of the $\zeta$-H\"older functions $h:M \longrightarrow \mathbb{R}$, i.e., if we define $$H_\zeta(h) := \sup _{x\neq y} \dfrac{|h(x) - h(y)|}{d_1(x,y)^\zeta},$$then

	$$H_\zeta:= \{ h:M \longrightarrow \mathbb{R}: H_\zeta(h) < \infty\}.$$ 	
\end{definition}

\begin{definition}\label{PH}
Let $\mathscr{P}_M$ be the set of H\"older potentials satisfying the following condition, which is open with respect to the H\"older norm:
\begin{enumerate}
	\item[(f3)] There exists $\epsilon_\varphi > 0$ sufficiently small such that
	\begin{equation}\label{f31}
		\sup \varphi - \inf \varphi < \epsilon_\varphi,
	\end{equation}
	and
	\begin{equation}\label{f32}
		H_\zeta(e^\varphi) < \epsilon_\varphi\, e^{\inf \varphi}.
	\end{equation}
\end{enumerate}
We assume that the constants $\epsilon_\varphi$ and $L$ satisfy
\begin{equation}\label{kdljfhkdjfkasd}
	\exp(\epsilon_\varphi)\,
	\frac{(\deg(f)-q)\sigma^{-\alpha} + q L^\alpha \bigl[1 + (L-1)^\alpha\bigr]}
	{\deg(f)} < 1.
\end{equation}
\end{definition}Equation (\ref{f32}) means that $\varphi$ belongs to a small cone of H\"older continuous functions (see \cite{VAC}).

\subsection{Spectral gap for $\mathcal{L}_\varphi$}

According to \cite{VAC}, a map $f:M \longrightarrow M$  (satisfying (f1), (f2) and (f3)) has an invariant probability $m$ of maximal entropy, absolutely continuous with respect to a conformal measure, $\nu$. And for every $\varphi \in \mathscr{P}_M$, its Ruelle-Perron-Frobenius (normalized by the maximal eigenvalue), $\mathcal{\overline{L}}_\varphi$ (see equation (\ref{hhdfhjdghf})) satisfies (P1) stated below, where $(B, | \cdot |_b)=(H_\zeta, |\cdot|_\zeta)$, with $|\varphi|_\zeta := H_\zeta (\varphi) + |\varphi|_{\infty} $ for all $\varphi \in H_\zeta$ and $(B_w, | \cdot |_w)=(C^0, |\cdot|_\infty)$.

\begin{enumerate}
	
	\item [(P1)] The following properties holds for $f$:

	\begin{enumerate}
		\item [(P1.1)] For every $\varphi \in \mathscr{P}_M$, there exists an equilibrium state $m(=m(\varphi))$, such that $m=h \nu$, where $0<\inf h\leq \sup h < + \infty$,  (see (P1.2)) $h, \frac{1}{h} \in B_s$ and $\mathcal{L}_\varphi (h)=\lambda h$, $\lambda >0$. Moreover, $\nu$ is a conformal measure with Jacobian $\lambda e^{-\varphi}$;

		\item  [(P1.2)] There exist normed spaces of real valued functions (not necessarily Banach), closed under the standard multiplication, $(B_s, |\cdot |_s) \subset (B_w, |\cdot|_w)$ satisfying $|\cdot|_w \leq |\cdot |_s$ such that for every $\varphi \in \mathscr{P}_M$ the spectrum of the Ruelle-Perron-Frobenius of $f:M \longrightarrow M$, $\mathcal{L}_\varphi:B_s \longrightarrow B_s$, defined by \begin{equation}\label{hhdfhjdghf}
			\mathcal{L}_\varphi(g)(x)= \sum_{y \in f^{-1}(x)} {g(y)e^{\varphi(y)}}, \ \ \forall g \in B_s
		\end{equation} is splitted as $$\func U\cup \{\lambda\},$$where $\func U$ is contained in a ball of radius smaller than $|\lambda|$.
		
		In this case, the action of the normalized operator ($\mathcal{\overline{L}}_\varphi:=\frac{1}{\lambda}\mathcal{L}_\varphi$) $\mathcal{\overline{L}}_\varphi: B_s \longrightarrow B_s$ can be decomposed as 
		\begin{equation}\label{utytsd}
			\mathcal{\overline{L}}_\varphi = \func {P} _f + \func {N}_f,
		\end{equation}where
		
		\begin{enumerate}
			\item $\func {P} _f^2= \func {P} _f$ and $\dim (\im (\func {P} _f))=1$; 
			\item there exist $0\leq r<1$ and $D>0$ such that $| \func {N}_f^n(g)|_s \leq Dr^n |g|_s$;
			\item $ \func {N}_f\func {P} _f = \func {P} _f \func {N}_f=0$. 
		\end{enumerate}Thus, if $\mathcal{L}_\varphi ^*$ denotes the dual of $\mathcal{L}_\varphi$ we have that $\mathcal{L}_\varphi ^*\nu = \lambda \nu$ and so 
		
		\begin{equation}\label{fixeddd}
			\mathcal{\overline{L}}_\varphi ^* \nu =  \nu,
		\end{equation}where $\mathcal{\overline{L}}_\varphi:=\dfrac{1}{\lambda}\mathcal{L}_\varphi$. 
		
	\end{enumerate}

\end{enumerate}Moreover, the following theorems \ref{LYgeral} and \ref{loiubb} hold, where the proof can be found in \cite{RR}.

\begin{theorem}\label{LYgeral}
	There exist constants $B_1>0$, $C_1>0$ and $0<\beta_1<1$ such that for all $%
	u \in B_s$, and all $n \geq 1$, it holds
	
	\begin{equation}
		|\mathcal{\overline{L}}_\varphi^n(u)|_s \leq B_1 \beta _1 ^n | u|_s + C_1|u|_{w}.
		\label{lasotaiiii}
	\end{equation}
\end{theorem}

\begin{theorem}\label{loiubb}
	There exist $0< r<1$ and $D>0$ s.t. for all $u \in \Ker(\func {P} _f)$, it holds $$|\mathcal{\overline{L}}_\varphi^n(u)|_s \leq Dr^n|u|_s \ \ \forall \ n \geq 1.$$
\end{theorem}

\begin{remark}\label{yturhfvb}
	We note that the Jacobian of the measure $m$ is $J_m(f) = \lambda \e ^{-\varphi} \frac{h\circ f}{h}$. Moreover, since $0<\inf h\leq \sup h < + \infty$ (see P1.1) the probabilities $m$ and $\nu$ are equivalents. Define $\mathcal{L}_{\varphi,h}:B_s \longrightarrow B_s$ by 
	
	\begin{equation}\label{jhfdgjhfg}
		\mathcal{L}_{\varphi,h} (u):= \frac{\mathcal{L}_{\varphi}(uh)}{h}
	\end{equation}
	for all $u \in B_s$. It is straightforward to see that $\mathcal{L}_{\varphi,h}$ satisfies (P1.1) and (P1.2), with $h \equiv 1$. In particular $\mathcal{L}_{\varphi,h}(1)=\lambda$. Let us denote its normalized version by  $\overline{\mathcal{L}}_{\varphi,h}: = \dfrac{\mathcal{L}_{\varphi,h}}{\lambda}$, in a way that $1$ is a fixed point for $\overline{\mathcal{L}}_{\varphi,h}$. 
\end{remark}

\begin{remark}\label{jhdfgjsd}
	The above theorems and properties also holds for the operator $\overline{\mathcal{L}}_{\varphi,h}$ defined in Remark \ref{yturhfvb}. In a way that, there are constants $B_3$, $\beta_3$, $C_3$, $r_3$ and $D_3$ such that for all $u \in B_s$

	\begin{equation}
		|\mathcal{\overline{L}}_{\varphi,h}^n(u)|_s \leq B_3 \beta _3 ^n | u|_s + B_3|u|_{w}.
		\label{lasotaiiiityrd}
	\end{equation}and for all $u \in \Ker (\func {P}_f)$ it holds
	
	\begin{equation}
		|\mathcal{\overline{L}}_{\varphi,h}^n(u)|_s \leq D_3r_3^n|u|_s \ \ \forall \ n \geq 1,
		\label{lasotaiisdffrr}
	\end{equation}where $\func {P}_f$ comes from de decomposition analogous to Equation (\ref{utytsd}) applied to $\mathcal{\overline{L}}_{\varphi,h}$.
	
\end{remark}

\subsection{Hypothesis on the fiber map $G$.}

We suppose that $G: \Sigma \longrightarrow K$ satisfies:

\begin{enumerate}
	\item [(H1)] $G$ is uniformly contracting on $m$-almost vertical fiber $\gamma_x :=\{x\}\times K$: there is $0 \leq \alpha <1$ such that for $m$-a.e. $x\in M$ it holds%
	\begin{equation}
		d_2(G(x,z_{1}),G(x,z_2))\leq \alpha d_2(z_{1},z_{2}), \quad \forall
		z_{1},z_{2}\in K.  \label{contracting1}
	\end{equation}
\end{enumerate}We denote the set of all vertical fibers $\gamma_x$, by $\mathcal{F}$: $$\mathcal{F}:= \{\gamma _x:=\{ x\}\times K; x \in M \} .$$ 

\begin{enumerate}
\item[(H2)] Let $P_1, \cdots, P_{\deg(f)}$ be the partition of $M$ given by (P2), and let $\zeta \leq 1$. Suppose that  
\begin{equation*}
	|G_i|_\zeta:= \sup _y\sup_{x_1, x_2 \in P_i} \dfrac{d_2(G(x_1,y), G(x_2,y))}{d_1(x_1,x_2)^\zeta}< \infty.
\end{equation*}
Denote by $|G|_\zeta$ the following constant: 
\begin{equation}\label{jdhfjdh}
	|G|_\zeta := \max_{i=1, \cdots, \deg(f)} \{|G_i|_\zeta\}.
\end{equation}
\end{enumerate}
\begin{remark}
	The condition (H2) implies that $G$ may be discontinuous on the sets $\partial P_i \times K$ for all $i=1, \cdots, \deg(f)$, where $\partial P_i$ denotes the boundary of $P_i$.
\end{remark}
\begin{remark}
	We note that elements of $\mathcal{F}^s$ i.e., $\gamma_x(=\{x\}\times K)$ for $x \in M$, are naturally identified with their "base point" $x$. For this reason, throughout this article, we will sometimes use the same notation for both without explicitly distinguishing them. In other words, the symbol $\gamma$ may refer either to an element of $\mathcal{F}^s$ or to a point in the set $M$, depending on the context. For instance, if $\phi:M \longrightarrow \mathbb{R}$ is a real-valued function, the expressions $\phi(x)$ and $\phi(\gamma)$ have the same meaning, as we are implicitly identifying $\gamma$ with $x$.
\end{remark}

The next Proposition \ref{kjdhkskjfkjskdjf} ensures the existence and uniqueness of an $F$-invariant measure, $\mu_0$, which projects on $m$. Since its proof is done by standard arguments (see \cite{AP}, for instance) we skip it. 

\begin{proposition}\label{kjdhkskjfkjskdjf}
	Let $m_1$ be an $f$-invariant probability. If $F$ satisfies (H1), then there exists an unique measure $\mu_1$ on $M \times K$ such that $\pi_1{_\ast}\mu_1 = m_1$ and for every continuous function $\psi \in C^0 (M \times K)$ it holds 
	\begin{equation*}
		\lim {\int{\inf_{\gamma \times K} \psi \circ F^n }dm_1(\gamma)}= \lim {\int{\sup_{\gamma \times K} \psi \circ F^n}dm_1 (\gamma)}=\int {\psi}d\mu_1. 
	\end{equation*}Moreover, the measure $\mu_1$ is $F$-invariant.
\end{proposition}Thus, if $F$ satisfies (f1), (f2), (f3) and (H1), it also satisfies (P3) stated below.

\begin{enumerate}
	\item [(P3)] There exists an $F$-invariant measure $\mu_0$, such that $\pi_{1*}\mu_0 =m$, where $\pi_{1}:\Sigma \rightarrow M$ is the projection defined by $\pi_{1}(x,y)=x$ and $\pi_{1*}$ denotes the pushforward map associated to $\pi_1$.
\end{enumerate}

\begin{remark}\label{pot}
	Note that, the probability $m$ is an equilibrium state for a potential $\varphi \in \mathscr{P}_M$. Thus $\mu_0$ depends on the function $\Phi = \varphi \circ \pi_1 \in \mathscr{P}_\Sigma$ (see equation (\ref{PPP})).
\end{remark}

\subsection{Examples}\label{dkjfhksjdhfksdf}

This section is devoted to several examples that demonstrate the breadth of our results. In Examples~\ref{uruitruytidjf}, \ref{skew}, \ref{poi} and \ref{poi2} we examine the base map $f$ or the fiber maps $G$ in isolation, rather than focusing on a specific skew-product coupling $(f, G)$. This modular approach implies that Theorems~A--I remain valid for any skew-product generated by such components. Conversely, Example~\ref{tsujii} provides a unified skew-product framework where our main results apply directly. Finally, Example~\ref{ferradura} illustrates a continuous dynamical setting characterized by uniform contraction across all fibers and the existence of a preserved horizontal fiber. This case underscores the utility of the non-trivial class $\mathcal{S}$ introduced in Section~\ref{SS}, which serves as a natural framework for such dynamics. Furthermore, the flexibility of this construction suggests that other examples demonstrating the applicability of $\mathcal{S}$ can be readily developed, particularly by considering different fiber maps or varying the geometry of the base map.

\begin{example}[Manneville--Pomeau map]\label{uruitruytidjf}
Let $\alpha \in (0,1)$ and consider the $C^{1+\alpha}$ local diffeomorphism $f_\alpha \colon [0,1] \to [0,1]$ given by
\[
f_{\alpha}(x)=
\begin{cases}
x\bigl(1+2^{\alpha}x^{\alpha}\bigr), & \text{for } 0 \le x \le \tfrac12,\\[2mm]
2x-1, & \text{for } \tfrac12 < x \le 1.
\end{cases}
\]The requirements (f1) and (f2) are straightforwardly met. Furthermore, a direct computation of the two inverse branches of $f_\alpha$ confirms that $L=1$.

To establish condition (f3), we examine the family of potentials $\{\varphi_{\alpha,t}\}_{t\in(-t_0,t_0)}$ defined as
\[
\varphi_{\alpha,t} = -t\log|Df_\alpha|,
\]
for a sufficiently small $t_0 > 0$. This yields a collection of $C^\alpha$ potentials such that $\{\varphi_{\alpha,t}\}_{t \in (-t_0,t_0)} \subset \mathscr{P}_M$ for each $\alpha \in (0,1)$. Indeed, for any $x,y \in [0,1]$, we observe that
\begin{eqnarray*}
|\varphi_{\alpha,t}(x)-\varphi_{\alpha,t}(y)|
&=& |t|\bigl|\log|Df_\alpha(x)|-\log|Df_\alpha(y)|\bigr|\\
&=& |t|\left|\log\frac{|Df_\alpha(x)|}{|Df_\alpha(y)|}\right|.
\end{eqnarray*}
Given that $|Df_\alpha|$ is bounded away from zero and infinity on the interval $(0,1]$ and exhibits at most polynomial growth as $x \to 0$, we can derive the uniform estimate
\[
\left|\log\frac{|Df_\alpha(x)|}{|Df_\alpha(y)|}\right| \le \log(2+\alpha),
\]
which implies
\[
|\varphi_{\alpha,t}(x)-\varphi_{\alpha,t}(y)| \le |t|\log(2+\alpha).
\]
It follows that for $|t|$ taken sufficiently small, the Hölder constant of $\varphi_{\alpha,t}$ satisfies the bound stipulated in (f3), thereby concluding the verification.
\end{example}

\begin{example} \label{skew}

Let $M := [0,1]^2$ be endowed with the metric $d_1((x_0, y_0), (x_1, y_1)) = \max \{|x_0-x_1|, |y_0-y_1|\}$. 
We consider a partition of $M$ into three vertical strips:
\[ P_0 = [0, 1/3] \times [0,1], \quad P_1 = [1/3, 2/3] \times [0,1], \quad P_2 = [2/3, 1] \times [0,1]. \]
Let $f: M \longrightarrow M$ be a skew-product map defined by:
\begin{equation}
    f(x,y) = (3x \pmod 1, \ g(x,y)),
\end{equation}
where the fiber map $g$ is $\delta$-close to the identity $\mathrm{id}(x,y)=y$ in the $C^2$ topology.

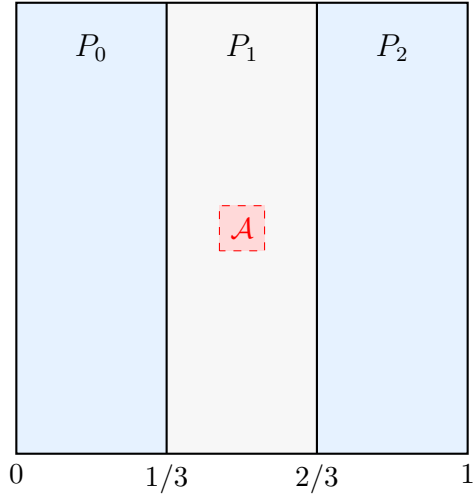
\begin{figure}[h!]
    \centering 
    
    \begin{tikzpicture}[scale=6] 
        \definecolor{expandingColor}{rgb}{0.9, 0.95, 1.0} 
        \definecolor{neutralColor}{rgb}{0.97, 0.97, 0.97} 
        \definecolor{criticalColor}{rgb}{1.0, 0.85, 0.85} 

        \fill[expandingColor] (0,0) rectangle (0.333, 1);
        \fill[expandingColor] (0.666,0) rectangle (1, 1);
        \fill[neutralColor] (0.333,0) rectangle (0.666, 1);
        
        \fill[criticalColor] (0.45, 0.45) rectangle (0.55, 0.55);
        \draw[dashed, red] (0.45, 0.45) rectangle (0.55, 0.55);
        \node[red, font=\bfseries] at (0.5, 0.5) {$\mathcal{A}$};

        \draw[thick] (0,0) rectangle (1,1); 
        \draw[thick] (0.333, 0) -- (0.333, 1); 
        \draw[thick] (0.666, 0) -- (0.666, 1); 

        \node[font=\large] at (0.166, 0.9) {$P_0$};
        \node[font=\large] at (0.5, 0.9)   {$P_1$};
        \node[font=\large] at (0.833, 0.9) {$P_2$};

        \node[below] at (0,0) {0};
        \node[below] at (0.333,0) {1/3};
        \node[below] at (0.666,0) {2/3};
        \node[below] at (1,0) {1};

    \end{tikzpicture}

    \caption{Partition of the domain $M$ into three vertical strips. The critical region $\mathcal{A}$ is entirely contained within the central strip $P_1$.}
    \label{fig:partition_map}
\end{figure}

The $C^2$ proximity to the identity implies the following bounds for the partial derivatives on $M$:
\begin{enumerate}
    \item $| \partial_y g(x,y) - 1 | \leq \delta$, for all $(x,y) \in M$;
    \item $| \partial_x g(x,y) | \leq \delta$, for all $(x,y) \in M$;
    \item $| \partial_{yy} g(x,y) | \leq \delta$, for all $(x,y) \in M$;
    \item $| \partial_{yx} g(x,y) | \leq \delta$, for all $(x,y) \in M$;
    \item \textbf{Hypothesis:} There exist $\varepsilon_0 > 0$ and $\varepsilon_1 > 1$ such that $|\partial_y g(x,y)| \geq \varepsilon_1$ for all $(x,y) \in \mathcal{A}^c$, where $\mathcal{A} := (1/2 - \varepsilon_0, 1/2 + \varepsilon_0)^2$.
\end{enumerate}

Moreover, we assume $g(1/2, 1/2) = 1/2$ and that $|\partial_y g(1/2, y_0)| \leq 1$ for some points $(1/2, y_0)$ in the critical region $\mathcal{A} := (1/2 - \varepsilon_0, 1/2 + \varepsilon_0)^2$, with $\varepsilon_0 > 0$.

\begin{tikzpicture}
\begin{axis}[
  axis lines = box,
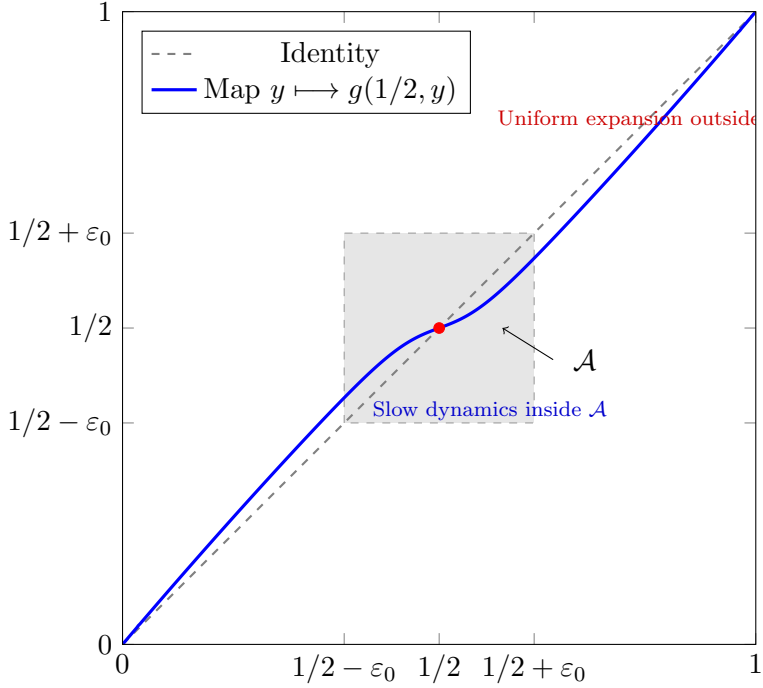
  xlabel = {Figure 2. The graph of $y \longmapsto g(1/2, y)$.},
  ymin=0, ymax=1, xmin=0, xmax=1,
  width=10cm, height=10cm,
  xtick={0, 0.35, 0.5, 0.65, 1},
  xticklabels={0, $1/2-\varepsilon_0$, $1/2$, $1/2+\varepsilon_0$, 1},
  ytick={0, 0.35, 0.5, 0.65, 1},
  yticklabels={0, $1/2-\varepsilon_0$, $1/2$, $1/2+\varepsilon_0$, 1},
  legend pos=north west,
  domain=0:1,
  samples=401,
  axis on top,
]

\fill [fill=gray!20] (axis cs:0.35,0.35) rectangle (axis cs:0.65,0.65);
\draw [gray!80,dashed] (axis cs:0.35,0.35) rectangle (axis cs:0.65,0.65);
\node[black,font=\bfseries] at (axis cs:0.73,0.45) {$\mathcal{A}$};
\draw[->,thin] (axis cs:0.68,0.45) -- (axis cs:0.60,0.5);

\pgfmathsetmacro{\lambda}{0.18}   
\pgfmathsetmacro{\kappa}{14}      

\addplot[dashed,black!50,thick] {x};
\addlegendentry{Identity}

\addplot[blue,very thick]
({x},
{
x
+ \lambda * x*(1-x)
  * tanh(\kappa*(0.5 - x))
});
\addlegendentry{Map $y \longmapsto g(1/2,y)$}

\addplot[mark=*,mark size=2pt,red,forget plot]
coordinates {(0.5,0.5)};

\node at (axis cs:0.58,0.37)
[font=\scriptsize,blue!80!black]
{Slow dynamics inside $\mathcal A$};

\node at (axis cs:0.80,0.83)
[font=\scriptsize,red!80!black]
{Uniform expansion outside};

\end{axis}

\end{tikzpicture}

In general, the derivative of $f$ is given by:
\begin{equation}
    Df(x,y) = \begin{pmatrix} 
    3 & 0 \\ 
    \partial_x g(x,y) & \partial_y g(x,y) 
    \end{pmatrix},
\end{equation}
for all $(x,y) \in M$. The determinant is simply $\det Df(x,y) = 3 \partial_y g(x,y)$. Moreover, the norm of the inverse derivative with respect to the $d_1$ metric is:
\begin{equation}
    L(x,y) = \|Df(x,y)^{-1}\| = \max \left\{ \frac{1}{3}, \ \frac{|\partial_x g(x,y)| + 3}{3|\partial_y g(x,y)|} \right\}.
\end{equation}
Thus, given the $C^2$ bounds on $g$, for all sufficiently small $\delta > 0$, we have:
\begin{equation}\label{ert}
    0 \leq L(x,y) \leq \frac{3 + \delta}{3(1 - \delta)}.
\end{equation}

The following lemma establishes the necessary and sufficient condition for $f$ to be a local expansion.

\begin{lemma}\label{lemma:expansion_condition}
The map $f$ is locally expanding at $(x,y)$ (i.e., $L(x,y) < 1$) if and only if the partial derivatives of the fiber map $g$ satisfy:
\begin{equation}
    |\partial_y g(x,y)| > 1 + \frac{|\partial_x g(x,y)|}{3}.
\end{equation}
\end{lemma}

\begin{proof}
Recall that $L(x,y) = \max \left\{ \frac{1}{3}, \ \frac{|\partial_x g(x,y)| + 3}{3 |\partial_y g(x,y)|} \right\}$. Since $1/3 < 1$, the condition $L(x,y) < 1$ is determined solely by the second term of the maximum. Therefore:
\begin{align*}
    L(x,y) < 1 &\iff \frac{|\partial_x g(x,y)| + 3}{3 |\partial_y g(x,y)|} < 1 \\
    &\iff |\partial_x g(x,y)| + 3 < 3 |\partial_y g(x,y)| \\
    &\iff 3 |\partial_y g(x,y)| > 3 + |\partial_x g(x,y)| \\
    &\iff |\partial_y g(x,y)| > 1 + \frac{|\partial_x g(x,y)|}{3}.
\end{align*}
\end{proof}

\begin{lemma}\label{lem:expansion_A_complement}
If $\varepsilon_1 = 1 + \delta/2$, then $f$ is a uniform expansion outside the region $\mathcal{A}$, i.e., $L(x,y) < 1$ for all $(x,y) \in \mathcal{A}^c$.
\end{lemma}

\begin{proof}
Assume $\varepsilon_1 = 1 + \delta/2$. According to the global bound $|\partial_x g(x,y)| \leq \delta$ and the hypothesis on $|\partial_y g|$ in $\mathcal{A}^c$, we have for all $(x,y) \in \mathcal{A}^c$:
\begin{equation*}
    |\partial_y g(x,y)| \geq 1 + \frac{\delta}{2} > 1 + \frac{\delta}{3} \geq 1 + \frac{|\partial_x g(x,y)|}{3}.
\end{equation*}
By Lemma \ref{lemma:expansion_condition}, the inequality $|\partial_y g| > 1 + \frac{1}{3}|\partial_x g|$ is precisely the necessary and sufficient condition for $L(x,y) < 1$. Consequently, this choice of $\varepsilon_1$ ensures that $f$ is uniformly expanding everywhere in the complement of $\mathcal{A}$.
\end{proof}

Now, let us analyze the geometric potential of the system. Define $\varphi := - \log |\det Df(x,y)|$, and note that $\varphi = \varphi(\delta)$ since the map $f$ depends on the parameter $\delta$. 

From the previous bounds, we have $1-\delta \leq \partial_y g(x,y) \leq 1+\delta$ for all $(x,y) \in M$. It follows that the determinant satisfies $3(1-\delta) \leq \det Df(x,y) \leq 3(1+\delta)$ for all $(x,y) \in M$. The oscillation of the potential $\varphi$ can be estimated as follows:
\begin{align*}
    \sup_{(x,y) \in M} \varphi - \inf_{(x,y) \in M} \varphi &= \sup_{(x,y) \in M} \left( \log \frac{1}{|\det Df(x,y)|} \right) - \inf_{(x,y) \in M} \left( \log \frac{1}{|\det Df(x,y)|} \right) \\
    &\leq \log \left( \frac{1}{3(1-\delta)} \right) - \log \left( \frac{1}{3(1+\delta)} \right) \\
    &= \log \left( \frac{3(1+\delta)}{3(1-\delta)} \right) = \log \left( \frac{1+\delta}{1-\delta} \right).
\end{align*}
Therefore, for all sufficiently small $\delta > 0$, we have:
\begin{equation}\label{dfdfds}
    \sup_{(x,y) \in M} \varphi - \inf_{(x,y) \in M} \varphi \leq \log \left( \frac{1+\delta}{1-\delta} \right).
\end{equation}

\noindent \textbf{Topological and Combinatorial Properties: (f1), (f2) and (P2):} The map $f$ is a local homeomorphism on each injectivity domain $\mathrm{int}(P_i)$. The first component, $3x \pmod 1$, maps each interval $[i/3, (i+1)/3]$ surjectively onto $[0,1]$. Since $g(x, \cdot)$ is a diffeomorphism of the fiber $[0,1]$ for every fixed $x$, it follows that $f(P_i) = [0,1] \times [0,1] = M$ for each $i \in \{0, 1, 2\}$. This confirms that $f$ is a full-branch map with degree $\deg(f) = 3$. 

Furthermore, the construction satisfies property (P2) for the partition $\{P_0, P_1, P_2\}$. Since the critical region $\mathcal{A}$ is centered at $(1/2, 1/2)$ and $\varepsilon_0$ is chosen small enough so that $\mathcal{A} \subset P_1 = [1/3, 2/3] \times [0,1]$, it follows that $\mathcal{A}$ is covered by only one partition element. Thus, we have $q = 1$, which satisfies the requirement $q < \deg(f)$.

\noindent \textbf{Regularity of the Potential and Gap Condition:} 

Let $J$ be the Lipschitz constant of the map $z \longmapsto \log z$ restricted to the range of the function $y \mapsto 3\partial_y g(x,y)$. Since $g$ is a $C^2$ function $\delta$-close to the identity (see items 1–4 above), the derivative $\partial_y g$ is Lipschitz with a constant proportional to the second derivative bound ($\delta$). Applying the Mean Value Theorem, there exists a constant $C > 0$ such that for any pair of points $z_0, z_1 \in M$:
\begin{align*}
    \frac{|\varphi(z_1) - \varphi(z_0)|}{d_1(z_1, z_0)^\zeta} 
    &= \frac{\left| \log |3\partial_y g(z_1)|^{-1} - \log |3\partial_y g(z_0)|^{-1} \right|}{d_1(z_1, z_0)^\zeta} \\
    &= \frac{\left| \log |\partial_y g(z_1)| - \log |\partial_y g(z_0)| \right|}{d_1(z_1, z_0)^\zeta} \\
    &\leq J \frac{\left| \partial_y g(z_1) - \partial_y g(z_0) \right|}{d_1(z_1, z_0)^\zeta} \\
    &\leq C \delta.
\end{align*}
Thus, the Hölder constant of the potential satisfies:
\begin{equation}\label{eq:holder_bound}
    H_\zeta (\varphi) \leq C \delta.
\end{equation}

On the other hand, considering the global bound on the determinant derived in the previous section, we have:
$$ e^{\inf \varphi} = \frac{1}{\sup |\det Df|} \geq \frac{1}{3(1+\delta)}. $$
Combining this estimate with Equation \eqref{eq:holder_bound}, we obtain:
\begin{equation}\label{eq:ratio_bound}
    \frac{H_\zeta (\varphi)}{e^{\inf \varphi}} \leq 3 C \delta (1+\delta).
\end{equation}
Since the right-hand side tends to 0 as $\delta \to 0$, the condition \eqref{f32} of Definition \ref{PH} is satisfied for sufficiently small $\delta$.

Finally, we verify the gap condition \eqref{kdljfhkdjfkasd}. Note that as $\delta \to 0$, we have $\epsilon_\varphi \to 0$ (implies $e^{\epsilon_\varphi} \to 1$) and $L \to 1$. Using $\deg(f)=3$, $q=1$, and recalling that $\sigma > 1$ represents the expansion factor outside $\mathcal{A}$, the inequality becomes:
\begin{equation}\label{eq:gap_check}
    e^{\epsilon_\varphi} \cdot \left( \frac{(\deg(f) - q)\sigma^{-\zeta} + q L^\zeta [1 + (L-1)^\zeta]}{\deg(f)} \right) 
    \longrightarrow \frac{2\sigma^{-\zeta} + 1}{3} < 1.
\end{equation}The limit is strictly less than 1 because $\sigma > 1 \implies \sigma^{-\zeta} < 1$.

The estimates in Equations \eqref{ert}, \eqref{dfdfds}, \eqref{eq:ratio_bound}, and \eqref{eq:gap_check} demonstrate that for a sufficiently small $\delta > 0$, the pair $(f, \varphi)$ satisfies all the required hypotheses: the structural conditions (f1), (f2) and (P2), and the regularity condition (f3).

\end{example}

\begin{example}[Discontinuous Maps: constant coefficients]\label{poi}
Let $I$ be the unit interval and $\mathcal{P}=\{P_i\}_{i=1}^{\deg(f)}$ its partition into open intervals $P_i = (a_i, b_i)$ as prescribed by (P2). Consider a collection of distinct real parameters $\{\alpha_i\}_{i=1}^{\deg(f)}$ satisfying $0 \le \alpha_i < 1$ for all $i$. We define the fiber map $G\colon I \times I \to I$ such that, for each $i \in \{1, \dots, \deg(f)\}$,
\[
G(x,y) = \alpha_i y, \quad \text{for all } x \in P_i \text{ and } y \in I.
\]
For $x$ outside the union of the interiors $\bigcup P_i$, we define $G$ such that it is left-continuous in the $x$-direction for every $i$. 

Given that the parameters $\alpha_i$ are distinct, $G$ necessarily exhibits jump discontinuities across the fibers $\{\partial P_i\} \times I$. Nevertheless, it straightforwardly fulfills condition~(H2). Indeed, as $G$ remains constant with respect to the base variable within each atom $P_i$, the variation $|G|_\zeta$ vanishes for every $\zeta>0$ (cf. Equation~\eqref{jdhfjdh}). Furthermore, the dynamics are uniformly contracting along the fibers with a global rate $\alpha = \max_{i} \{\alpha_i\} < 1$, thereby satisfying condition~(H1).
\end{example}

\begin{example}[Discontinuous Maps: Lipschitz coefficients]\label{poi2} 
Let $I$ be the unit interval and $\mathcal{P}=\{P_i\}_{i=1}^{\deg(f)}$ its partition into open intervals $P_i = (a_i, b_i)$ as prescribed by (P2). Consider a constant $0 \leq \alpha < 1$ and a finite collection of Lipschitz functions $\{h_i\}_{i=1}^{\deg(f)}$, where each $h_i \colon \overline{P}_i \to [0, \alpha]$ satisfies the jump condition $h_i(b_i) \neq h_{i+1}(a_{i+1})$ at the boundaries. We further assume that the Lipschitz constants are uniformly bounded, denoting $L := \sup_{i} L(h_i) < \infty$.

We define the fiber map $G \colon I \times I \to I$ such that, for each $i \in \{1, \dots, \deg(f)\}$,
\[
G(x, y) = h_i(x) y, \quad \text{for all } x \in P_i \text{ and } y \in I.
\]
For $x$ outside the union of the interiors $\bigcup P_i$, we extend the definition of $G$ to be left-continuous in the $x$-direction for every $i$. 

By design, the jump conditions on $h_i$ ensure that $G$ exhibits discontinuities across the fibers $\{\partial P_i\} \times I$. Condition~(H2) remains satisfied as a direct consequence of the uniform bound on the Lipschitz constants. Furthermore, since $G$ operates as a uniform $\alpha$-contraction on each fiber, condition~(H1) is also readily fulfilled.
\end{example}

\begin{example}{(Fat solenoidal attractors)}\label{tsujii}
	Consider the class of dynamical systems defined by $$F:S^1 \times \mathbb{R} \longrightarrow S^1 \times \mathbb{R} \ \ F(x,y)=(lx, \alpha y + o(x)),$$where $l \geq 2$ is an integer, $0< \alpha < 1$ is a real number, and $o$ is a $C^2$ function on the unit circle $S^1$. In \cite{Tsujii}, M. Tsujii proved the existence of an ergodic probability measure $\mu$ on $S^1 \times \mathbb{R}$ such that Lebesgue almost every point is generic. That is, $$\lim _{n \rightarrow \infty } \dfrac{1}{n} \sum _{i=0}^{n-1} \delta _{F^i(x)} = \mu \ \ \textnormal{weakly.}$$This measure $\mu$ is called the SRB measure for $F$. In the same work, Tsujii characterizes the regularity of $\mu$ with respect to Lebesgue measure depending on the parameters that define $F$.

Fix an integer $l\geq 2$. Define $\mathcal{D} \subset (0,1) \times C^2(S^1, \mathbb{R})$ as the set of pairs $(\alpha, o)$ for which the SRB measure $\mu$ is absolutely continuous with respect to the Lebesgue measure on $S^1 \times \mathbb{R}$. Let $\mathcal{D}^o \subset \mathcal{D}$ denote the interior of $\mathcal{D}$ taken with respect to the product topology given by the standard topology on $(0,1)$ and the $C^2$-topology on $C^2(S^1,\mathbb{R})$. The following result is Theorem 1 on page 1012 of \cite{Tsujii}.

\begin{theorem}
	Let $l^{-1}< \lambda < 1$. There exists a finite collection of $C^\infty$ functions $u_i: S^1 \longrightarrow \mathbb{R}$, $i=1, \cdots m$, such that for any $C^2$ function $g \in C^2(S^1, \mathbb{R})$, the subset of $\mathbb{R}^m$ $$\left \{ (t_1, t_ 2, \cdots, t_m) \in \mathbb{R}^m | \left(\alpha, g(x) + \sum _{i=1}^{m} t_ iu_i (x)\right) \notin \mathcal {D}^o\right\}$$is a null set with respect to the Lebesgue measure on $\mathbb{R}^m$.
\end{theorem}

We now apply our results to show not only that we can construct equilibrium states possessing all the statistical and analytical properties described in the applications section, but also that these measures coincide with the SRB measure studied by Tsujii in \cite{Tsujii}.

Let us consider the setting: $M=S^1$ (unit circle), $K=[-2,2]$, $f=lx$, $G(x,y)=\alpha y + o(x)$, $\varphi = -\log |f'|$ and $\Phi = \varphi \circ \pi _1$. Since $f'$ is constant, it follows that $\Phi \in \mathscr{P}_\Sigma$ and $\varphi \in \mathscr{P}_M$, according to Definitions \ref{P} and \ref{PH}. Therefore, there exists a unique equilibrium state $\mu_0 \in \mathbf{S}^\infty$ for the pair $(F, \Phi)$. Moreover, all Theorems A through D apply to $(F,\mu_0)$. Additionally, since $\varphi = -\log |f'|$, the associated conformal measure $\nu$ coincides with Lebesgue measure $m$ on $S^1$. This implies $\pi_{1*}\mu_0 \ll m$, and by Theorem A of \cite{VAC}, $\pi_{1*}\mu_0$ is ergodic. Moreover, the measure $\mu_0$, whose existence is guaranteed by Theorem~\ref{belongsss}, coincides with the physical measure $\mu$ constructed in \cite{Tsujii} (refer to \cite{RR} for further details).

\end{example}

The next example was introduced in \cite{DHRS}, where it was shown that the non-wandering set of $F$ is partially hyperbolic for certain fixed parameters satisfying the conditions stated below. This class of maps was subsequently studied in \cite{LOR}, \cite{RA2}, \cite{RA1}, and \cite{RA3}. In \cite{RA2}, the inverse map $F^{-1}$ was considered and shown to admit a skew product structure, where the base dynamics is strongly topologically mixing and non-uniformly expanding, while the fibre dynamics is uniformly contracting. In \cite{RA3}, the author established the existence of equilibrium states and proved stability results for this class of systems.

\begin{example}{(Partially hyperbolic horseshoes)} \label{ferradura} Consider the cube $R= [0,1]\times[0,1]\times[0,1]\subset\mathbb{R}^3$  and   the parallelepipeds
		$$R_0 =[0,1]\times [0,1]\times [0,1/6]\quad \mbox{and} \quad R _1=[0,1]\times [0,1]\times [5/6,1].$$ 
		Consider a map   defined for $(x,y,z)\in R_0$ as  
		$$ F_{0}(x,y,z) =(\rho x , f(y),\beta z),$$
		where $0 < \rho <{1/3}$, $\beta> 6$ and  $$f(y) =\frac {1}{1 - \left(1-\frac{1}{y}\right)e^{-1}}.$$
		Consider also a map  defined for $(x,y,z)\in R_1$ as
		$$F_{1}(x,y,z)  = \left(\frac{3}{4}- \rho x , \sigma (1 - y) ,\beta_{1} \left(z - \frac{5}{6} \right)\right),$$
		where  $0<\sigma< {1/3}$ and $3< \beta_1 < 4$.
		Define the horseshoe map $F$ on $R$ as
		$$F\vert_{R_0}=F_0,\quad F\vert_{R_1}=F_1,
		$$
		with $R\setminus(R_0\cup R_1)$ being mapped injectively outside $R$. In \cite{DHRS} it was proved that the non-wandering set of $F$ is partially hyperbolic when we consider fixed parameters satisfying conditions above.  
	\end{example}

\section{Preliminary Concepts}\label{seccc}

\subsection{The W-K like norm}
Let $(X,d)$ be a compact metric space, $u:X\longrightarrow \mathbb{R}$ be a
$\zeta$-H\"older function, and $H_\zeta(u)$ be its best $\zeta$-H\"older's constant. That is,
\begin{equation}\label{lipsc}
	\displaystyle{H_\zeta(u)=\sup_{x,y\in X,x\neq y}\left\{ \dfrac{|u(x)-u(y)|}{d(x,y)^\zeta}%
		\right\} }.
\end{equation}In what follows, we present a generalization of the Wasserstein-Kantorovich-like metric given in \cite{GLu}. 
\begin{definition}
	Given two signed measures, $\mu $ and $\nu $ on $X,$ we define the \textbf{\
		Wasserstein-Kantorovich-like} distance between $\mu $ and $\nu $ by 
	\begin{equation*}
		W_{1}^{\zeta}(\mu ,\nu ):=\sup_{H_\zeta(u)\leq 1,|u|_{\infty }\leq 1}\left\vert \int {\
			u}d\mu -\int {u}d\nu \right\vert .
	\end{equation*}Since  $\zeta$ is a constant,  we denote%

	\begin{equation}
		||\mu ||_{W}:=W_{1}^{\zeta}(0,\mu ),  \label{WW}
	\end{equation}and observe that $||\cdot ||_{W}$ defines a norm on the vector space of signed measures defined on a compact metric space. It is worth remarking that this norm is equivalent to the standard norm of the dual space of $\zeta$-H\"older functions. 
	\label{wasserstein}
\end{definition}

\subsection{Rokhlin's Disintegration Theorem.}

Consider a probability space $(\Sigma,\mathcal{B}, \mu)$ and a partition $%
\Gamma$ of $\Sigma$ into measurable sets $\gamma \in \mathcal{B}$. Denote by $%
\pi : \Sigma \longrightarrow \Gamma$ the projection that associates to each
point $x \in M$ the element $\gamma _x$ of $\Gamma$ that contains $x$. That is, 
$\pi(x) = \gamma _x$. Let $\widehat{\mathcal{B}}$ be the $\sigma$-algebra of 
$\Gamma$ provided by $\pi$. Precisely, a subset $\mathcal{Q} \subset \Gamma$
is measurable if, and only if, $\pi^{-1}(\mathcal{Q}) \in \mathcal{B}$. We
define the \textit{quotient} measure $\mu _x$ on $\Gamma$ by $\mu _x(%
\mathcal{Q})= \mu(\pi ^{-1}(\mathcal{Q}))$.

The proof of the following theorem can be found in \cite{Kva}, Theorem
5.1.11 (items a), b) and c)) and Proposition 5.1.7 (item d)).

\begin{theorem}
	(Rokhlin's Disintegration Theorem) Suppose that $\Sigma $ is a complete and
	separable metric space, $\Gamma $ is a measurable partition of $\Sigma $ and $\mu $ is a probability on $\Sigma $. Then, $\mu $ admits a
	disintegration relative to $\Gamma $. That is, there exists a family $\{\mu _{\gamma}\}_{\gamma \in \Gamma }$ of probabilities on $\Sigma $ and a quotient measure $\mu _{x}$, such that:
	
	\begin{enumerate}
		\item[(a)] $\mu _\gamma (\gamma)=1$ for $\mu _x$-a.e. $\gamma \in \Gamma$;
		
		\item[(b)] for all measurable set $E\subset \Sigma $ the function $\Gamma
		\longrightarrow \mathbb{R}$ defined by $\gamma \longmapsto \mu _{\gamma
		}(E), $ is measurable;
		
		\item[(c)] for all measurable set $E\subset \Sigma $, it holds $\mu (E)=\int 
		{\mu _{\gamma }(E)}d\mu _{x}(\gamma )$.

		\label{rok}
		\item [(d)] If the $\sigma $-algebra $\mathcal{B}$ on $\Sigma $ has a countable
		generator, then the disintegration is unique in the following sense. If $(\{\mu _{\gamma }^{\prime }\}_{\gamma \in \Gamma },\mu _{x})$ is another disintegration of the measure $\mu $ relative to $\Gamma $, then $\mu
		_{\gamma }=\mu _{\gamma }^{\prime }$, for $\mu _{x}$-almost every $\gamma
		\in \Gamma $.
	\end{enumerate}
\end{theorem}

\subsection{The Space $\mathbf{AB}_m$.}

Let $\mathcal{SB}(\Sigma )$ be the space of Borel  signed measures on $\Sigma : = M \times K$. Given $\mu \in \mathcal{SB}(\Sigma )$ denote by $\mu ^{+}$ and $\mu ^{-}$
the positive and the negative parts of its Jordan decomposition, $\mu =\mu
^{+}-\mu ^{-}$ (see Remark {\ref{ghtyhh}). Let $\pi _1:\Sigma
	\longrightarrow M$ \ be the projection defined by $\pi_1 (x,y)=x$, denote
	by $\pi _{1\ast }:$}$\mathcal{SB}(\Sigma )\rightarrow \mathcal{SB}(M)${\
	the pushforward map associated to $\pi _1$.

	Denote by $\mathbf{AB}_m$ the
	set of signed measures $\mu \in \mathcal{SB}(\Sigma )$ such that its
	associated positive and negative marginal measures, $\pi _{1\ast }\mu ^{+}$
	and $\pi _{1\ast }\mu ^{-},$ are absolutely continuous with respect to $m$, i.e.,
	\begin{equation}
		\mathbf{AB}_m=\{\mu \in \mathcal{SB}(\Sigma ):\pi _{1\ast }\mu ^{+}<<m\ \ 
		\mathnormal{and}\ \ \pi _{1\ast }\mu ^{-}<< m\}.  \label{thespace1}
	\end{equation}%
}Given a \emph{probability measure} $\mu \in \mathbf{AB}_m$ on $\Sigma $,
Theorem \ref{rok} describes a disintegration $\left( \{\mu _{\gamma
}\}_{\gamma },\mu _{x}\right) $ along $\mathcal{F}$ by a family $\{\mu _{\gamma }\}_{\gamma }$ of probability measures
on the stable leaves and, since 
$\mu \in \mathbf{AB}_m$, $\mu _{x}$ can be identified with a non negative
marginal density $\phi _{1}:M\longrightarrow \mathbb{R}$, defined almost
everywhere, with $|\phi _{1}|_{1}=1$. \ For a general (non normalized)
positive measure $\mu \in \mathbf{AB}_m$ we can define its disintegration in
the same way. In this case, $\mu _{\gamma }$ are still probability measures, $%
\phi _{1}$ is still defined and $\ |\phi _{1}|_{1}=\mu (\Sigma )$.


\begin{definition}
	Let $\pi _{2}:\Sigma \longrightarrow K$ be the projection defined by $%
	\pi _{2}(x,y)=y$. Let $\gamma \in \mathcal{F}$, consider $\pi
	_{\gamma ,2}:\gamma \longrightarrow K$, the restriction of the map $\pi
	_{2}:\Sigma \longrightarrow K$ to the vertical leaf $\gamma $, and the
	associated pushforward map $\pi _{\gamma ,2\ast }$. Given a positive measure 
	$\mu \in \mathbf{AB}_m$ and its disintegration along the stable leaves $%
	\mathcal{F}$, $\left( \{\mu _{\gamma }\}_{\gamma },\mu _{x}=\phi
	_{1} m \right)$, we define the \textbf{restriction of $\mu $ on $\gamma $}
	and denote it by $\mu |_{\gamma }$ as the positive measure on $K$ (not
	on the leaf $\gamma $) defined, for all mensurable set $A\subset K$, as 
	\begin{equation*}
		\mu |_{\gamma }(A)=\pi _{\gamma ,2\ast }(\phi _{1}(\gamma )\mu _{\gamma
		})(A).
	\end{equation*}%
	For a given signed measure $\mu \in \mathbf{AB}_m $ and its Jordan
	decomposition $\mu =\mu ^{+}-\mu ^{-}$, define the \textbf{restriction of $%
		\mu $ on $\gamma $} by%
	\begin{equation*}
		\mu |_{\gamma }=\mu ^{+}|_{\gamma }-\mu ^{-}|_{\gamma }.
	\end{equation*}%
	\label{restrictionmeasure}
\end{definition}

\begin{remark}
	\label{ghtyhh}As proved in Appendix 2 of \cite {GLu}, the restriction $%
	\mu |_{\gamma }$ does not depend on the decomposition. Precisely, if $\mu
	=\mu _{1}-\mu _{2}$, where $\mu _{1}$ and $\mu _{2}$ are any positive
	measures, then $\mu |_{\gamma }=\mu _{1}|_{\gamma }-\mu _{2}|_{\gamma }$ $%
	m$-a.e. $\gamma \in M$. Moreover, as showed in \cite{GLu}, the restriction is linear in the sense that $(\mu_1 + \mu_2)|_\gamma = \mu_1|_\gamma + \mu_2|_\gamma$.
\end{remark}

\subsection{H\"older-Measures.}

A positive measure on $M \times K$ can be disintegrated along the stable leaves $\mathcal{F}^s$ in such a way that we can regard it as a family of positive measures on $M$, denoted by $\{\mu |_\gamma\}_{\gamma \in \mathcal{F}^s}$. Since there exists a one-to-one correspondence between $\mathcal{F}^s$ and $M$, this defines a path in the metric space of positive measures ($\mathcal{SB}(K)$) defined on $K$, represented by $M \longmapsto \mathcal{SB}(K)$, where $\mathcal{SB}(K)$ is equipped with the Wasserstein-Kantorovich-like metric (see Definition \ref{wasserstein}). In this article, we use a functional notation and denote such a path by $\Gamma_{\mu}: M \longrightarrow \mathcal{SB}(K)$, defined almost everywhere by $\Gamma_{\mu}(\gamma) = \mu|_\gamma$, where $(\{\mu_{\gamma}\}_{\gamma \in M}, \phi_{1})$ is some disintegration of $\mu$.
However, since this disintegration is defined $\widehat{\mu}$-a.e. $\gamma \in M$, the path $\Gamma_\mu$ is not unique. For this reason, we define $\Gamma_{\mu}$ as the class of almost everywhere equivalent paths corresponding to $\mu$.

\begin{definition}\label{defd}
	Consider a positive Borelean measure $\mu \in \mathbf{AB}$, and a disintegration  $\omega=(\{\mu _{\gamma }\}_{\gamma \in M},\phi
	_1)$, where $\{\mu _{\gamma }\}_{\gamma \in M }$ is a family of
	probabilities on $M \times K$ defined $\widehat{\mu}$-a.e. $\gamma \in M$ (where $\widehat{\mu} := \pi_1{_*}\mu=\phi _1 m$) and $\phi
	_1:M\longrightarrow \mathbb{R}$ is a non-negative marginal density. Denote by $\Gamma_{\mu }$ the class of equivalent paths associated to $\mu$ 
	\begin{equation*}
		\Gamma_{\mu }=\{ \Gamma^\omega_{\mu }\}_\omega,
	\end{equation*}
	where $\omega$ ranges on all the possible disintegrations of $\mu$ and $\Gamma^\omega_{\mu }: M\longrightarrow \mathcal{SB}(K)$ is the map associated to a given disintegration, $\omega$:
	$$\Gamma^\omega_{\mu }(\gamma )=\mu |_{\gamma } = \pi _{\gamma, 2} ^\ast \phi _1
	(\gamma)\mu _\gamma .$$We denote the set on which $\Gamma_{\mu }^\omega $ is defined by $M_{\omega} \left( \subset M\right)$.
\end{definition}

\begin{definition}For a given $0<\zeta <1$, a disintegration $\omega$ of $\mu$, and its functional representation $\Gamma_{\mu }^\omega $, we define the \textbf{$\zeta$-H\"older constant of $\mu$ associated to $\omega$} by

	\begin{equation}\label{Lips1}
		|\mu|_\zeta ^\omega := \esssup _{\gamma_1, \gamma_2 \in M_{\omega}} \left\{ \dfrac{||\mu|_{\gamma _1}- \mu|_{\gamma _2}||_W}{d_1 (\gamma _1, \gamma _2)^\zeta}\right\},
	\end{equation}where the essential supremum is taken with respect to $m$. Finally, we define the \textbf{$\zeta$-H\"older constant} of the positive measure $\mu$ by

	\begin{equation}
		\label{Lips2}
		|\mu|_\zeta :=\displaystyle{\inf_{ \Gamma_{\mu }^\omega \in \Gamma_{\mu } }\{|\mu|_\zeta ^\omega\}}.
	\end{equation}
	
	\label{Lips3}
\end{definition}

\begin{remark}
	When no confusion is possible, to simplify the notation, we denote $\Gamma_{\mu }^\omega (\gamma )$ just by $\mu |_{\gamma } $.
\end{remark}

\begin{definition}\label{sdfsdjhfjhsgjdfgjsd}
	From the Definition \ref{Lips3} we define the set of the $\zeta$-H\"older measures, $\mathcal{H} ^\zeta_m$, as
	\begin{equation}
		\mathcal{H} ^\zeta _m=\{\mu \in \mathbf{AB}: |\mu |_\zeta <\infty \}.
	\end{equation}
\end{definition}

\begin{proposition}\label{ttty}
	For all $\mu \in \mathcal{H} ^\zeta _m$, the following inequality holds: 
	\begin{equation}\label{uit}
		H_\zeta(\phi_{1}) \leq |\mu |_\zeta.
	\end{equation}
\end{proposition}

\begin{proof}
	Let $\omega = (\{\mu _ \gamma \}_{\gamma \in M}, \phi_1)$ be a disintegration of a measure $\mu \in \mathcal{H} ^\zeta _m$. By Equation (\ref{WW}) and Definition \ref{restrictionmeasure}, we obtain
	\begin{eqnarray*}
		|\phi_{1}(\gamma _1) - \phi_{1}(\gamma _2)| &=& |\mu |_{\gamma _1}(\Sigma) - \mu |_{\gamma _2}(\Sigma)|\\&=&|\int 1 d \mu |_{\gamma _1} - \int 1 d\mu |_{\gamma _2}| \\&\leq &||\mu |_{\gamma _1} - \mu |_{\gamma _2}||_W.
	\end{eqnarray*}Dividing both sides by $d_1 (\gamma _1, \gamma _2)^\zeta$ and taking the essential supremum, we obtain $$H_\zeta(\phi_{1}) \leq |\mu|_\zeta ^\omega.$$ Taking the infimum over all possible disintegrations $\omega$ completes the proof.
\end{proof}

\subsection{The $\mathbf{L}_m^{\infty}$ and $\mathbf{S}^\infty$ spaces.}\label{jdfjdhkjf}

\begin{definition}\label{linffff}
	Let $\mathbf{L}_m^{\infty } \subseteq \mathbf{AB}_m(\Sigma )$ be defined as%
	\begin{equation*}
		\mathbf{L}_m^{\infty}=\left\{ \mu \in \mathbf{AB}_m:\esssup (||\mu
		^{+}|_{\gamma }-\mu ^{-}|_{\gamma }||_W<\infty \right\},
	\end{equation*}%
	where the essential supremum is taken over $M$ with respect to $m$.
	Define the function $||\cdot ||_{\infty }:\mathbf{L}_m^{\infty
	}\longrightarrow \mathbb{R}$ by%
	\begin{equation*}
		||\mu ||_{\infty }=\esssup ||\mu ^{+}|_{\gamma }-\mu ^{-}|_{\gamma
		}||_W.
	\end{equation*}
\end{definition}

\begin{definition}\label{sinf}
	Finally, consider the following set of signed measures on $\Sigma $%
	\begin{equation}\label{sinfi}
		\mathbf{S}^{\infty }=\left\{ \mu \in \mathbf{L}^{\infty };\phi _{1}\in
		H_\zeta \right\},
	\end{equation}%
	and the function, $||\cdot ||_{\mathbf{S}^{\infty }}:\mathbf{S}^{\infty }\longrightarrow 
	\mathbb{R}$, defined by%
	\begin{equation*}
		||\mu ||_{\mathbf{S}^{\infty }}=|\phi _{1}|_\zeta+||\mu ||_{\infty },
	\end{equation*} where $|\varphi|_\zeta := H_\zeta (\varphi) + |\varphi|_{\infty} $ for all $\varphi \in H_\zeta$.
\end{definition}

\section{The actions on $\mathcal{H} ^\zeta_m$, $\mathbf{L}^\infty _m$ and $\mathbf{S}^\infty$.} \label{from}

\subsection{The operator}In this section, we define the linear operator $\func{F} _{\Phi,h}:\mathbf{AB}_m \longrightarrow \mathbf{AB}_m$ by selecting appropriate expressions for the disintegration of the measure $\func{F} _{\Phi,h} \mu$. To achieve this, we also consider the transfer operator of $F$, defined by the standard expression $\func F_\ast \mu (A) = \mu(F ^{-1}(A))$ for any measure $\mu$ and any measurable set $A$.

In what follows, we define a set of fibre constant potentials defined on $\Sigma$.

\begin{definition}\label{P}
Define
\begin{equation}\label{PPP}
	\mathscr{P}_\Sigma
	:= \left\{ \Phi \; ; \; \Phi = \varphi \circ \pi_1,
	\ \varphi \in \mathscr{P}_M \right\}.
\end{equation}
\end{definition}Moreover, from now and ahead $\chi _{A}$ stands for the characteristic function of $A$ and $\gamma _i$ denotes the $i$-th pre-image of $\gamma \in f^{-1}(\gamma)$, for all $i=1, \cdots, \deg(f)$ and all $\gamma \in M$. Also define, for a given $\gamma \in \mathcal{F}^s$, the map $F_{\gamma }:K\longrightarrow K$ by 
\begin{equation}\label{ritiruwt}
	F_{\gamma }=\pi _{2}\circ F|_{\gamma }\circ \pi _{\gamma ,2}^{-1},
\end{equation}where $\pi _{2} (x,y) = y$ and $\pi _{\gamma , 2}$ is the restriction of $\pi _{2}$ on the leaf $\gamma$.

All results and definitions in this section are taken from Subsection~2.1.2 of \cite{RR}. For this reason, their proofs are omitted.

\begin{definition}\label{fphi}
	For a given potential $\Phi \in \mathscr{P}_\Sigma$, define $\func {F_{\Phi,h}}:\mathbf{AB}_m \longrightarrow \mathbf{AB}_m$ (where $h$ is given by (P1.1)) as the linear operator such that for all $\mu \in \mathbf{AB}_m$

	\begin{equation}
		(\func{F}_{\Phi,h} \mu )_{x}:=\mathcal{L}_{\varphi,h}(\phi _{1})m  \label{1}
	\end{equation}
	and
	\begin{equation}\label{2}
		(\func{F} _{\Phi,h} \mu )_{\gamma }=\frac{1}{h(\gamma)\mathcal{L}_{\varphi,h}(\phi
			_{1})(\gamma )}\sum_{i=1}^{\deg (f)}{(\phi _{1} (\gamma _i) h (\gamma_i) e^{ \varphi(\gamma _i)})\cdot \func{F}_{\ast
			}\mu _{\gamma _i}}\cdot \chi _{f(P_i)} (\gamma)
	\end{equation}
	when $\mathcal{L}_{\varphi,h}(\phi _{1})(\gamma )\neq 0$. Otherwise, if $\mathcal{L}_{\varphi,h}(\phi _{1})(\gamma )=0$, then $(\func{F} _{\Phi,h} \mu )_{\gamma }$ is the Lebesgue measure
	on $\gamma$ (it could be defined with any other measure instead of Lebesgue). 
\end{definition}

\begin{corollary}\label{fff}
	For every $\mu \in \mathbf{AB}_m$, it holds	$$(\func{F} _{\Phi,h} \mu)|_\gamma= \frac{1}{h(\gamma)}\sum _{i=1}^{\deg(f)}{\func {F}_{\gamma_i*}\mu|_{\gamma_i}h(\gamma_i)e^{\varphi(\gamma_i)}}$$(see Equation (\ref{ritiruwt})) for $m$-a.e. $\gamma \in M$.
\end{corollary}

\begin{definition}\label{normalized}
	Define the operator $\func {\overline{F}}_{\Phi,h}$, by $$\func {\overline{F}}_{\Phi,h}:=\dfrac{1}{\lambda} \func {F}_{\Phi,h}.$$ In this article, we will refer to this operator as the \textbf{transfer operator of} $(F, \Phi)$ or simply as the \textbf{transfer operator} when the potential $\Phi$ is clear from the context.
\end{definition}

\begin{corollary}\label{disintnorm}
	For every $\mu \in \mathbf{AB}_m$, the restriction of the measure $\func {\overline{F}}_{\Phi,h} \mu$ to the leaf $\gamma$, $(\func {\overline{F}}_{\Phi,h} \mu )|_\gamma$, is given by the expression
	$$ (\func {\overline{F}}_{\Phi,h} \mu )|_\gamma = \dfrac{1}{\lambda} (\func{F} _{\Phi,h} \mu)|_\gamma,$$for $m$-a.e. $\gamma \in M$.	
\end{corollary}

\begin{remark}\label{jhdgflasçlçl}
	Since there is no risk of confusion, we will henceforth denote the operators $\func{F}_{\Phi,h}$ and $\func {\overline{F}}_{\Phi,h}$ simply as $\func{F}_{\Phi}$ and $\func {\overline{F}}_{\Phi}$, respectively. Moreover, for simplicity, the notations $\mathbf{AB}_m$, $\mathbf{L}^{\infty}_m$ and $\mathcal{H} ^\zeta_m$ will be abbreviated as $\mathbf{AB}$, $\mathbf{L}^{\infty}$ and $\mathcal{H} ^\zeta$, respectively.
\end{remark}

\subsection{On the actions of $\func {\overline{F}}_{\Phi}$ on $\mathcal{H}^\zeta _m$, $\mathbf{L}^\infty$ and $\mathbf{S}^\infty$.} \label{from2}

\subsubsection{On the actions of $\func {\overline{F}}_{\Phi}$ on $\mathbf{L}^\infty$ and $\mathbf{S}^\infty$.}\label{from1}

The propositions \ref{kdjfjksdkfhjsdfk} and \ref{lasotaoscilation22} are taken from Subsection~2.1.3 of \cite{RR}. For this reason, their proofs are omitted. We draw the reader's attention to the abbreviations used in the notation introduced in Remark~\ref{jhdgflasçlçl}.

\begin{proposition}\label{kdjfjksdkfhjsdfk}
	The operator $\func {\overline{F}}_{\Phi}: \mathbf{L}^{\infty} \longrightarrow \mathbf{L}^{\infty}$ is a weak contraction. It holds, $||\func {\overline{F}}_{\Phi} \mu ||_\infty \leq ||\mu||_\infty$, for all $\mu \in \mathbf{L}^{\infty}$.
\end{proposition}

\begin{proposition}[Lasota-Yorke inequality for $\mathbf{S}^\infty$] 
	For all $\mu
	\in \mathbf{S}^\infty$, it holds%
	\begin{equation}
		||\func {\overline{F}}_{\Phi}^{n}\mu ||_{\mathbf{S}^\infty}\leq A\beta _2 ^{n}||\mu
		||_{\mathbf{S}^\infty}+B_2||\mu ||_\infty,\ \ \forall n\geq 1.  \label{xx}
	\end{equation}
	\label{lasotaoscilation22}
\end{proposition}

\subsubsection{On the action of $\func {\overline{F}}_{\Phi}$ on $\mathcal{H}^\zeta _m$}\label{from2}

For the next lemma, for a given path $\Gamma _\mu$ which represents the measure $\mu$, we define for each $\gamma \in I_{\Gamma_{\mu }^\omega }\subset M$, the map

\begin{equation}
	\mu _F(\gamma) := \func{F_\gamma }_*\mu|_\gamma,
\end{equation}where $F_\gamma :K \longrightarrow K$ is defined as

\begin{equation}\label{poier}
	F_\gamma (y) = \pi_2 \circ F \circ {(\pi _2|_\gamma)} ^{-1}(y)
\end{equation}and $\pi_2 : M\times K \longrightarrow  K$ is the second coordinate projection $\pi_2(x,y)=y$.

Lemmas~\ref{niceformulaac} and~\ref{opsdas} are precisely Lemmas 3.6 and 3.7 in~\cite{RR}, and their proofs are omitted.

\begin{lemma}
	\label{niceformulaac} For every $\mu \in \mathbf{AB}$ and $m$-a.e. stable leaf $%
	\gamma \in \mathcal{F}^s$, it holds 
	\begin{equation}
		||\func{F}_{\gamma \ast }\mu |_{\gamma }||_W\leq ||\mu |_{\gamma }||_W,
		\label{weak1}
	\end{equation}%
	where $F_{\gamma }:K\longrightarrow K$ is defined in Equation (\ref{ritiruwt}) and $\func{F}_{\gamma \ast }$ is the associated pushforward
	map. Moreover, if $\mu $ is a probability measure on $K$, it holds 
	\begin{equation}
		||\func{F}_{\gamma\ast}^{n}\mu ||_W=||\mu ||_W=1,\ \ \forall \ \
		n\geq 1.  \label{simples}
	\end{equation}
\end{lemma}

\begin{lemma}\label{opsdas}
	For all signed measures $\mu $ on $K$ and for $m$-a.e. $\gamma \in \mathcal{F}^s%
	$, it holds%
	\begin{equation*}
		||\func{F}_{\gamma \ast }\mu ||_W\leq \alpha^\zeta ||\mu ||_W+|\mu (K)|
	\end{equation*}%
	($\alpha $ is the rate of contraction of $G$, see \eqref{contracting1}). In
	particular, if $\mu (K)=0$ then%
	\begin{equation*}
		||\func{F}_{\gamma \ast }\mu ||_W\leq \alpha^\zeta ||\mu ||_W.
	\end{equation*}%
	\label{quasicontract}
\end{lemma}

A proof of the following result is given in Lemma~3.21 of~\cite{RR}.

\begin{lemma}\label{apppoas}
	Suppose that $F:\Sigma \longrightarrow \Sigma$ satisfies (H1) and (H2). Then, for all $\mu \in \mathcal{H} _\zeta^{+} $ which satisfy $\phi _1 = 1$ $m$-a.e., it holds $$||\func{F}%
	_{x  \ast }\mu |_{x  } - \func{F}%
	_{y \ast }\mu |_{y  }||_W \leq \alpha^\zeta |\mu|_\zeta  d_1(x, y)^\zeta  + |G|_\zeta d_1(x, y)^\zeta ||\mu ||_\infty,$$ for all $x,y \in P_i$ and all $i=1, \cdots, \deg(f)$.
\end{lemma}

For the next proposition and henceforth, for a given path $\Gamma _\mu ^\omega \in \Gamma_{ \mu }$ (associated with the disintegration $\omega = (\{\mu _\gamma\}_\gamma, \phi _1)$, of $\mu$), unless written otherwise, we consider the particular path $\Gamma_{\func{\overline{F}}_\Phi\mu} ^\omega \in \Gamma_{\func{\overline{F}}_\Phi \mu}$ defined by the corollaries \ref{fff} and \ref{disintnorm}, by the expression

\begin{equation}
	\Gamma_{\func{\overline{F}}_\Phi \mu} ^\omega (\gamma)=\dfrac{1}{\lambda h(\gamma)}\sum_{i=1}^{\deg(f)}{\func{F}%
		_{\gamma _i \ast }\Gamma _\mu ^\omega (\gamma_i)h(\gamma_i)e^{\varphi(\gamma_i)}}\ \ m \mathnormal{-a.e.}\ \ \gamma \in M.  \label{niceformulaaareer}
\end{equation}Recall that $\Gamma_{\mu} ^\omega (\gamma) = \mu|_\gamma:= \pi_{2*}(\phi_{1}(\gamma)\mu _\gamma)$ and in particular $\Gamma_{\func{\overline{F}}_\Phi\mu} ^\omega (\gamma) = (\func{\overline{F}}_\Phi\mu)|_\gamma = \pi_{2*}(\mathcal{\overline{L}}_\varphi\phi_1(\gamma)(\func{F}_\Phi\mu )_\gamma)$, where $\phi_1 = \dfrac{d \pi _{1*} \mu}{dm}$.

The following Proposition~\ref{iuaswdas} and Corollary~\ref{kjdfhkkhfdjfh} correspond to Proposition~3.23 and Corollary~3.24 in~\cite{RR}, and their proofs are therefore omitted.

\begin{proposition}\label{iuaswdas}
	If $F:\Sigma \longrightarrow \Sigma$ satisfies (f1), (f2), (f3), (H1) and (H2) and $(\alpha \cdot L)^\zeta<1$, then there exist $0<\beta<1$ and $D >0$, such that for all $\mu \in \mathcal{H} _\zeta^{+} $ which satisfy $\phi _1 = 1$ $m$-a.e. and for all $\Gamma ^\omega _\mu \in \Gamma _\mu$, it holds $$|\Gamma_{\func{\overline{F}}_\Phi } ^\omega\mu|_{\zeta}  \leq \beta |\Gamma_{\mu}^\omega|_\zeta + D||\mu||_\infty,$$ for $\beta:= (\alpha L)^\zeta$ and $D:=\{\epsilon _\varphi L^\zeta + |G|_ \zeta L^\zeta\}$.
\end{proposition}

By iterating the inequality $|\Gamma_{\func{\overline{F}}_\Phi \mu}^\omega|_{\zeta}  \leq \beta |\Gamma_{\mu}^\omega|_\zeta + D||\mu||_\infty$ obtained in Proposition \ref{iuaswdas}, along with a standard computation, we arrive at the following result, the proof of which is omitted.

\begin{corollary}\label{kjdfhkkhfdjfh}
	Suppose that $F:\Sigma \longrightarrow \Sigma$ satisfies (f1), (f2), (f3), (H1), (H2) and $(\alpha \cdot L)^\zeta<1$. Then, for all $\mu \in \mathcal{H} _\zeta^{+} $ which satisfy $\phi _1 = 1$ $m$-a.e., it holds 
	\begin{equation}\label{erkjwr}
		|\Gamma_{\func{\overline{F}}_\Phi^n\mu}^\omega|_{\zeta}  \leq \beta^n |\Gamma _\mu^\omega|_\zeta + \dfrac{D}{1-\beta}||\mu||_\infty,
	\end{equation}
	for all $n\geq 1$, where $\beta$ and $D$ are from Proposition \ref{iuaswdas}.
\end{corollary}

\begin{remark}\label{kjedhkfjhksjdf}
	Taking the infimum over all paths $\Gamma_{ \mu } ^\omega  \in \Gamma_{ \mu }$ and all $\Gamma_{\func{\overline{F}}_\Phi^n\mu}^\omega  \in \Gamma_{\func{\overline{F}}_\Phi^n\mu}$ on both sides of inequality (\ref{erkjwr}), we get 
	
	\begin{equation}\label{fljghlfjdgkdg}
		|\func{\overline{F}}_\Phi ^n\mu|_{\zeta}  \leq \beta^n |\mu|_\zeta + \dfrac{D}{1-\beta}||\mu||_\infty. 
	\end{equation}The above equation (\ref{fljghlfjdgkdg}) will give a uniform bound (see the proof of Theorem \ref{riirorpdf}) for the H\"older's constant of the measure $\func{\overline{F}}_\Phi^{n} m_1$, for all $n$. Where $m_1$ is defined as the product $m_1=m \times m_2$, for a fixed probability measure $m_2$ on $K$. The uniform bound will be useful later on (see Theorem \ref{disisisii}).
	
\end{remark}

\begin{remark}\label{riirorpdf}
	Let $m_1$ be the measure from Remark \ref{kjedhkfjhksjdf}. Consider its trivial disintegration, denoted by $\omega_0 =(\{m_{1 \gamma}  \}_{\gamma}, \phi_1)$, where $m_{1 \gamma}$ is given by $$m_{1 \gamma} = \func{\pi _{2,\gamma}^{-1}{_*}}m_2 \ \ \forall \gamma$$ and $\phi _1 \equiv 1$. By this definition, we have $$m_1|_\gamma = m_2, \ \ \forall \ \gamma.$$ As a consequence, $\Gamma ^{\omega _0}_{m_1}$ is constant, meaning $$\Gamma ^{\omega _0}_{m_1} (\gamma)= m_2 \ \ \forall \gamma.$$ This implies that $m_1 \in \mathcal{H} _\zeta^{+}$. More generally, for any $n \in \mathbb{N}$, let $\omega_n$ denote the specific disintegration of $\func{\overline{F}}_\Phi^nm_1$ defined from $\omega_0$ via Equation (\ref{niceformulaaareer}). For a given $n$, we refer to the corresponding path as $\Gamma_{\func{\overline{F}^n}_\Phi m_1} ^{\omega_n}$, which will be used multiple times throughout this article.
\end{remark}

Theorem~\ref{belongsss} below corresponds to Theorem~H in \cite{RR}. We omit its proof here, as it can be found in \cite[Subsection~3.1.5]{RR}. This result is fundamental, as it characterizes the conditions under which the probability measure $\mu_0$ is an equilibrium state.

\begin{theorem}\label{belongsss}
Assume that $F$ satisfies conditions (f1), (f2), (f3), and (H1). 
Then, for each potential $\Phi \in \mathscr{P}_\Sigma$, there exists a unique $F$-invariant measure $\mu_0 \in \mathbf{S}^1$, which in addition belongs to $\mathbf{S}^\infty$. In particular, if $F$ uniformly contracts all fibres, then $\mu_0$ is an equilibrium state.
\end{theorem}

The following result provides an estimate for the regularity of the invariant measure of $F$. Its proof can be found in Theorem I of \cite{RR}. This type of result has numerous applications, some of which are also discussed in \cite{RRRSTAB}, \cite{GLu}, \cite{R}, and \cite{DR}.

\begin{athm}\label{regg}
Suppose that $F:\Sigma \to \Sigma$ satisfies conditions (f1), (f2), (f3), (H1), (H2), and $(\alpha \cdot L)^\zeta < 1$.
For each potential $\Phi \in \mathscr{P}_\Sigma$, let $\mu_0$ denote the $F$-invariant measure of Theorem \ref{belongsss}.
Then $\mu_0 \in \mathcal{H}_\zeta^{+}$ and
\begin{equation}\label{VA}
|\mu_0|_\zeta \leq \frac{D}{1-\beta},
\end{equation}
where $D$ and $\beta$ are the constants given in Proposition~\ref{iuaswdas}.
\end{athm}

Theorem~\ref{quasiquasiquasi} below coincides with Theorem~A of \cite{RR}. The proof appears in Subsection~2.2 of \cite{RR}, and Subsection~3.1.4 of that work verifies that Theorem~A is applicable to the system introduced in Section~3 of \cite{RR}, which is the same system considered here. Therefore, we omit the proof.

Let us consider the following sets of zero average measures in $\mathbf{S}^{\infty}$ defined by 
\begin{equation}
	\mathbf{V}^\infty=\{\mu \in \mathbf{S}^{\infty}:\phi _1 \in \Ker(\func {P} _f)\}. \label{mathVV}
\end{equation}

\begin{theorem}[Exponential convergence to the equilibrium in $\mathbf{L}^{\infty}$ and $\mathcal{S}^{\infty}$] \label{5.8} Suppose that $F$ satisfies (f1), (f2), (f3), (H1) and (H2). Then, there exist $D_3\in \mathbb{R}$ and $0\leq \beta _3<1$ such that
	for every signed measure $\mu \in \mathbf{V}^\infty$, it holds 
	\begin{equation*}
		||\func {\overline{F}}_{\Phi,h}^{n}\mu ||_{\infty}\leq D_3\beta _3^{n}||\mu ||_{\mathbf{S}^{\infty}},
	\end{equation*}%
	for all $n\geq 1$, where $\beta _3=\max \{\sqrt{r},%
	\sqrt{\alpha}\}$ and $D_3=(\sqrt{\alpha }^{-1}+\overline{\alpha }_1D \sqrt{r}^{-1})$.\label{quasiquasiquasi}
\end{theorem}

\section{Exponential Decay of Correlations and Central Limit Theorem}\label{limit}

In this section, we establish several limit theorems for the dynamical system $(F,\mu_0)$. 
In particular, we show that the system exhibits exponential decay of correlations for observables in the space $L^1(\mathcal{F}_0)$, introduced below, and in the H\"older space $\ho_\zeta(\Sigma)$. 
Moreover, we prove that $(F,\mu_0)$ satisfies the Central Limit Theorem for H\"older observables.

\subsection{Exponential Decay of Correlations over the constant fibre functions}

Throughout this section, the map $F:\Sigma \to \Sigma$ is assumed to satisfy conditions (f1), (f2), (f3), (H1), and (H2), and we assume that $(\alpha \cdot L)^\zeta < 1$.

We prove that $(F,\mu_0)$ exhibits exponential decay of correlations for two classes of observables: 
the first consists of functions that are constant along the fibres, and the second is the space of H\"older continuous functions. To accomplish this, we need the following lemma.

\begin{lemma}\label{uiytirut}
	Let $\varphi \in \ho_\zeta(\Sigma)$ be a H\"older function on $\Sigma$. 
	Then there exists a disintegration $(\{\mu_{0,\gamma}\}_{\gamma \in M}, \phi_1)$ of $\mu_0$ such that 
	the real-valued function 
	\[
	\gamma \longmapsto \int \varphi|_{\gamma} \, d\mu_{0,\gamma}
	\]
	is a H\"older function on $M$.
\end{lemma}

\begin{proof}
	Since the H\"older constant of $\mu_0$ is finite (Theorem~\ref{regg}), 
	there exists a disintegration $\omega = (\{\mu_{0,\gamma}\}_{\gamma \in M}, \phi_1)$ of $\mu_0$ 
	such that $|\Gamma_{\mu_0}^{\omega}|_\zeta < +\infty$. 
	Moreover, $\phi_1 \equiv 1$ $m$-a.e.
	
	Consider a H\"older function $\varphi \in \ho_\zeta(\Sigma)$. 
	For each $\gamma \in M$, the restriction 
	$\varphi(\gamma, \cdot): K \to \mathbb{R}$ is a H\"older function on $K$ 
	with H\"older constant $\ho_\zeta(\varphi_\gamma)$. 
	Moreover, $\ho_\zeta(\varphi_\gamma) \le \ho_\zeta(\varphi)$ for all $\gamma$, 
	where $\ho_\zeta(\varphi)$ denotes the H\"older constant of $\varphi$.
	
	Thus, by Theorem~\ref{regg}, we have
	\begin{align*}
		\Big| \int \varphi|_{\gamma_1} \, d\mu_{0,\gamma_1} 
		- \int \varphi|_{\gamma_2} \, d\mu_{0,\gamma_2} \Big|
		&\le 
		\Big| \int \varphi|_{\gamma_1} \, d\mu_{0,\gamma_1}
		- \int \varphi|_{\gamma_1} \, d\mu_{0,\gamma_2} \Big|  \\
		&\quad +
		\Big| \int \varphi|_{\gamma_1} \, d\mu_{0,\gamma_2}
		- \int \varphi|_{\gamma_2} \, d\mu_{0,\gamma_2} \Big|  \\
		&\le 
		\max\{\ho_\zeta(\varphi), \|\varphi\|_\infty\}
		\|\mu _{0}|_{\gamma _1} -  \mu _{0}|_{\gamma _2} \|_W  \\
		&\quad +
		\ho_\zeta(\varphi) \, d_1(\gamma_1, \gamma_2)^{\zeta}
		\int 1 \, d\mu_{0,\gamma_2}  \\
		&\le 
		\max\{\ho_\zeta(\varphi), \|\varphi\|_\infty\}
		\, d_1(\gamma_1, \gamma_2)^{\zeta} |\Gamma_{\mu_0}^{\omega}|_\zeta
		\\
		&\quad +
		\ho_\zeta(\varphi) \, d_1(\gamma_1, \gamma_2)^{\zeta}.
	\end{align*}
	
	This completes the proof.
\end{proof}

The following proposition is a direct consequence of Theorem~4.8 in \cite{VAC}. For this reason, we omit its proof.

\begin{proposition}\label{iutryrt}
	There exists a constant $0 < \tau_2 < 1$ such that, for all $\psi \in L^1_{m}(M)$ and all $\varphi \in \ho_\zeta(M)$, it holds
	\[
	\left| 
	\int (\psi \circ f^n)\, \varphi \, dm 
	- \int \psi \, dm \int \varphi \, dm 
	\right| 
	\le \tau_2^n D(\psi, \varphi) \quad \forall\, n \ge 1,
	\]
	where $D(\psi, \varphi) > 0$ is a constant depending on $\psi$ and $\varphi$.
\end{proposition}

\begin{athm}\label{çljghhjçh}
	Suppose that $F:\Sigma \longrightarrow \Sigma$ satisfies (f1), (f2), (f3), (H1) and (H2) and $(\alpha \cdot L)^\zeta<1$. Let $\mu_0$ be the unique $F$-invariant probability in $\mathbf{S}^\infty$. There exists a constant $0 < \tau_2 < 1$ such that, for every constant fibre function 
	$\psi: \Sigma \to \mathbb{R}$ satisfying $\psi(\cdot, y) \in L^1(m)$ for all $y$, and for every 
	$\varphi \in \ho_\zeta(\Sigma)$, we have
	\[
	\left| 
	\int (\psi \circ F^n)\, \varphi \, d\mu_0 
	- \int \psi \, d\mu_0 \int \varphi \, d\mu_0 
	\right| 
	\le \tau_2^n D(\psi, \varphi) \quad \forall\, n \ge 1,
	\]
	where $D(\psi, \varphi) > 0$ is a constant depending on $\psi$ and $\varphi$.
\end{athm}

\begin{proof}
	Let $\psi: \Sigma \to \mathbb{R}$ be such that $\psi(x,y) = \psi(x)$ for all $y \in K$, and assume 
	$\psi(\cdot, y) \in L^1(m)$. 
	Let $\varphi \in \ho_\zeta(\Sigma)$ be a H\"older function.
	
	Since $\psi$ depends only on the $x$-coordinate, we have 
	$\psi \circ F^n = \psi \circ f^n$ for all $n$, and 
	\[
	\int \psi(x,y)\, d\mu_0(x,y) = \int \psi(x)\, dm(x).
	\]
	Denote by 
	\[
	s(\gamma) := \int_K \varphi|_{\gamma}\, d\mu_{0,\gamma}
	\]
	the real-valued function introduced in Lemma~\ref{uiytirut}.
	Then, by that result, $s$ is a H\"older function on $M$.
	
	Hence, by Proposition \ref{iutryrt}, we have
	\begin{align*}
		\left| 
		\int (\psi \circ F^n)\, \varphi \, d\mu_0 
		- \int \psi \, d\mu_0 \int \varphi \, d\mu_0 
		\right|
		&= 
		\left| 
		\int_M \int_K (\psi \circ f^n)\, \varphi \, d\mu_{0,\gamma}\, dm 
		- \int_M \psi\, dm \int_M \int_K \varphi\, d\mu_{0,\gamma}\, dm 
		\right| \\
		&= 
		\left| 
		\int_M (\psi \circ f^n) \int_K \, \varphi \, d\mu_{0,\gamma}\, dm 
		- \int_M \psi\, dm \int_M \int_K \varphi\, d\mu_{0,\gamma}\, dm 
		\right| \\
		&= 
		\left| 
		\int_M (\psi \circ f^n)\, s\, dm 
		- \int_M \psi\, dm \int_M s\, dm
		\right| \\
		&\le \tau_2^n D(\psi, \varphi) \quad \forall\, n \ge 1.
	\end{align*}
\end{proof}

\subsection{Exponential Decay of Correlations over $L^1(\mathcal{F}_0)$}\label{f0}

In this section, we introduce the space $L^1(\mathcal{F}_0)$, on which we will later establish the exponential decay of correlations. 
This result will have a further application, where we prove the Central Limit Theorem for H\"older observables.

Denote by \(\mathcal{B}=\mathcal{B}(M)\) the Borel \(\sigma\)-algebra of $M$.
Let \(\mathcal{G}\) be a countable generator of \(\mathcal{B}\) and let \(\mathcal{A}\) be the algebra generated by \(\mathcal{G}\). Define
\[
\mathcal{A}\times K:=\{\,A\times K \subset \Sigma \;:\; A\in\mathcal{A}\,\}.
\]
Then \(\mathcal{A}\times K\) is an algebra of subsets of $\Sigma$.

Finally, let \(\mathcal{F}_0\) be the \(\sigma\)-algebra generated by \(\mathcal{A}\times K\). Then \(\mathcal{F}_0\) is contained in the $\sigma$-algebra where $\mu_0$ is defined.

\begin{proposition}\label{nvbvdjkf}
	Let $\gamma \in \mathcal{F}^s$ be a stable leaf and $A$ a measurable subset of $\mathcal{F}_0$. If $\gamma \cap A \neq \emptyset$, then $\gamma \subset A$.
\end{proposition}

\begin{proof}
	For a fixed $\gamma \in \mathcal{F}^s$, let $\mathcal{C}$ be the collection defined by
	\[
	\mathcal{C} := \{\, A \in \mathcal{F}_0 \;:\; A \cap \gamma \neq \emptyset \Longrightarrow \gamma \subset A \,\}.
	\]
	It is straightforward to see that $\mathcal{C}$ is a monotone class that contains the algebra $\mathcal{A} \times K$. Hence, $\mathcal{C} \supset \mathcal{F}_0$. Since this holds for all $\gamma$, the proof is complete.
\end{proof}

\begin{athm}\label{athmc}
	Suppose that $F:\Sigma \longrightarrow \Sigma$ satisfies (f1), (f2), (f3), (H1) and (H2) and $(\alpha \cdot L)^\zeta<1$. Let $\mu_0$ be the unique $F$-invariant probability in $\mathbf{S}^\infty$. There exists a constant $0 < \tau _3 < 1$ such that, for all $\psi \in L^1_{\mu_0}(\mathcal{F}_0)$ and all $\varphi \in \ho_\zeta(\Sigma)$, it holds
	\[
	\left| 
	\int (\psi \circ F^n)\, \varphi \, d\mu_0 
	- \int \psi \, d\mu_0 \int \varphi \, d\mu_0 
	\right| 
	\le \tau_3^n D(\psi, \varphi) \quad \forall\, n \ge 1,
	\]
	where $D(\psi, \varphi) > 0$ is a constant depending on $\psi$ and $\varphi$.
\end{athm}

\begin{proof}
	
	Note that if $\psi$ is a real-valued $\mathcal{F}_0$-measurable function, then $\psi$ is constant along the fibres.  
	Indeed, $\psi^{-1}(x)$ is $\mathcal{F}_0$-measurable. If $\psi^{-1}(x)\neq\emptyset$, then there exist $\gamma\in\mathcal{F}^s$ and $z\in\gamma$ such that $\psi(z)=x$.  
	By Proposition~\ref{nvbvdjkf}, we have $\gamma \subset \psi^{-1}(x)$, and therefore $\psi(z)=x$ for all $z\in\gamma$.
	
	Now fix $\gamma_0\in\mathcal{F}^s$ and let $x_0\in\psi(\gamma_0)$.  
	Then $\psi^{-1}(x_0)\cap\gamma_0\neq\emptyset$.  
	Thus, by the argument above, $\gamma_0 \subset \psi^{-1}(x_0)$, which shows that $\psi$ is constant on $\gamma_0$.
	
	Moreover, if $\psi\in L^1(\mu_0)$, then for every $y\in K$ we have $\psi(\cdot,y)\in L^1(m)$.
	
	By Theorem \ref{çljghhjçh}, the proof is finished.
\end{proof}

\subsection{The Gordin's Theorem}

We now introduce an important tool that will be used, together with Theorem~\ref{athmc}, to establish the Central Limit Theorem: Gordin's theorem. Its proof can be found in \cite{Stoc}, and therefore we omit it here.

Let $(\Sigma,\mathcal{F}_0,\mu)$ be a probability space where $\mathcal{F}_0$ is a $\sigma$-algebra of $\Sigma$. Let $F:\Sigma \longrightarrow \Sigma$ be a $\mathcal{F}_0$-measurable function and let $\mathcal{F}_n := F^{-n}(\mathcal{F}_0)$ be a non-increasing family of $\sigma$-algebras. 
A function $\xi : M \to \mathbb{R}$ is $\mathcal{F}_n$-measurable if and only if there exists a $\mathcal{F}$-measurable function $\xi_n$ such that $\xi = \xi_n \circ F^n$. 
Define 
\[
L^2(\mathcal{F}_n) = \left\{\, \xi \in L^2(\mu) \; ; \; \xi \text{ is } \mathcal{F}_n\text{-measurable} \,\right\}.
\]
Note that $L^2(\mathcal{F}_{n+1}) \subset L^2(\mathcal{F}_n)$ for each $n \ge 0$. 
Given $\varphi \in L^2(\mu)$, we denote by $\mathbb{E}(\varphi \mid \mathcal{F}_n)$ the $L^2$-orthogonal projection of $\varphi$ onto $L^2(\mathcal{F}_n)$.

\begin{theorem}[Gordin] \label{TeoremaDeGordin}
	Let $(\Sigma, \mathcal{F}_0, \mu)$ be a probability space, and let $\phi \in L^2(\mu)$ satisfy $\int \phi \, d\mu = 0$. 
	Assume that $F : \Sigma \to \Sigma$ is a measurable map and that $\mu$ is an $F$-ergodic invariant probability measure. 
	Let $\mathcal{F}_0 \subset \mathcal{F}$ be such that $\mathcal{F}_n := F^{-n}(\mathcal{F}_0)$, $n \in \mathbb{N}$, forms a nonincreasing family of $\sigma$-algebras. Define
	\[
	\sigma_\phi^2 := \int \phi^2 \, d\mu + 2 \sum_{j=1}^{\infty} \int \phi \, (\phi \circ F^j) \, d\mu.
	\]
	Assume that
	\[
	\sum_{n=0}^\infty \left\| \mathbb{E}(\phi \mid \mathcal{F}_n) \right\|_2 < \infty . 
	\quad 
	\]
	Then $\sigma_\phi < \infty$, and $\sigma_\phi = 0$ if and only if $\phi = u \circ F - u$ for some $u \in L^2(\mu)$. 
	Moreover, if $\sigma_\phi > 0$, then for any interval $A \subset \mathbb{R}$,
	\[
	\mu\!\left(\, x \in \Sigma : \frac{1}{\sqrt{n}} 
	\sum_{j=0}^{n-1} \phi(F^j(x)) \in A \right)
	\to 
	\frac{1}{\sigma_\phi \sqrt{2\pi}} 
	\int_A e^{-\frac{t^2}{2\sigma_\phi^2}} \, dt,
	\quad \text{as } n \to \infty.
	\]
\end{theorem}

\subsection{Application: Central Limit Theorem}
In order to apply Gordin's theorem, we consider the sigma $\sigma$-algebra $\mathcal{F}_0$ introduced in Section \ref{f0}.

\begin{lemma}\label{previous}
	For every H\"older continuous function $\phi$ satisfying $\int \phi \, d\mu = 0$, 
	there exists a constant $R = R(\phi)$ such that
	\[
	\|\mathbb{E}(\phi \mid \mathcal{F}_n)\|_2 \le R \tau^n
	\quad \text{for all } n \ge 0.
	\]
\end{lemma}

\begin{proof}
	By the Theorem \ref{athmc}, if $\psi \in L^1_{\mu_0}(\mathcal{F}_0)$ with $\int \psi \, d\mu \le 1$, then
	\[
	\left| \int (\psi \circ F^n)\phi \, d\mu_0 - \int \psi \, d\mu_0 \int \phi \, d\mu_0 \right|
	\le D(\psi, \phi)\tau_3^n.
	\]
	
	Since $\|\psi\|_1 \le \|\psi\|_2$ and $\int \phi \, d\mu = 0$, we obtain
	\[
	\begin{aligned}
		\|\mathbb{E}(\phi \mid \mathcal{F}_n)\|_2
		&= \sup \left\{ \int \xi \phi \, d\mu : \xi \in L^2(\mathcal{F}_n), \|\xi\|_2 = 1 \right\} \\
		&= \sup \left\{ \int (\psi \circ f^n) \phi \, d\mu : \psi \in L^2(\mathcal{F}_0), \|\psi\|_2 = 1 \right\}
		\\ & \le R(\phi)\tau_3^n.
	\end{aligned}
	\]
\end{proof}

\begin{athm}[\textbf{Central Limit Theorem}]\label{central}
	Suppose that $F:\Sigma \longrightarrow \Sigma$ satisfies (f1), (f2), (f3), (H1) and (H2) and $(\alpha \cdot L)^\zeta<1$. Let $\mu_0$ be the unique $F$-invariant probability in $\mathbf{S}^\infty$. 
	Given a H\"older continuous function $\varphi$, define
	\[
	\sigma_\varphi^2 := \int \phi^2 \, d\mu + 2 \sum_{j=1}^{\infty} \int \phi \, (\phi \circ F^j) \, d\mu,
	\quad \text{where } \phi = \varphi - \int \varphi \, d\mu.
	\]
	Then $\sigma_\varphi < \infty$ and $\sigma_\varphi = 0$ if and only if 
	$\varphi = u \circ F - u$ for some $u \in L^2(\mu)$. 
	Moreover, if $\sigma_\varphi > 0$, then for every interval $A \subset \mathbb{R}$,
	\[
	\lim_{n \to \infty}
	\mu\left(x \in M : 
	\frac{1}{\sqrt{n}} \sum_{j=0}^{n-1}
	\left( \varphi(F^j(x)) - \int \varphi \, d\mu \right)
	\in A
	\right)
	= \frac{1}{\sigma_\varphi \sqrt{2\pi}} \int_A e^{-\frac{t^2}{2\sigma_\varphi^2}} \, dt.
	\]
\end{athm}

\begin{proof}
	By the Lemma \ref{previous}, 
	$\sum_{n=0}^\infty \|\operatorname{E}(\phi \mid \mathcal{F}_n)\|_2 < \infty$, 
	so the condition of Gordin's Theorem holds and it completes the proof.
\end{proof}

\section{Statistical Stability}\label{loeritu}

\subsection{ Admissible $R(\delta)$-Perturbations }\label{kjrthkje}

We define an \textbf{admissible $R(\delta)$-perturbation} as a family $\{F_{\delta }\}_{\delta \in [0,1)}$, where $F_{\delta }$ satisfies the following conditions (U1), (U2), (U3), (A1), (A2), as well as (f1), (f2), (f3), (H1), and (H2) for all $\delta$.

\begin{enumerate}
	\item [(U1)] There exists a small enough $\delta _1 $ such that for all $\delta \in (0, \delta_1)$,  it holds
	$$\displaystyle{ \deg (f_\delta)=\deg(f)},$$ for all $\delta \in (0, \delta_1)$;
	
\end{enumerate}

\begin{enumerate}
	
	\item [(U2)] For every $\gamma \in M$ and for all $i=1,\cdots,\deg(f)$ denote by $\gamma_{\delta,i}$ the $i$-th pre-image of $\gamma$ by $f_\delta$. Suppose there exists a real-valued function $\delta \longmapsto R(\delta) \in \mathbb{R}^+$ such that $$\lim_{\delta \rightarrow 0^+} {R(\delta)\log (\delta)}=0$$ and the following three conditions hold:

	\begin{enumerate}
		\item [(U2.1)] For a given $\delta \in [0,\delta_1)$ and $\gamma \in M$ define 
		\begin{equation}\label{iueyrtfd}
		g_ \delta(\gamma):=\dfrac{h_ \delta(\gamma)e^{\varphi(\gamma)}}{\lambda _ \delta h_ \delta\circ f_ \delta (\gamma)}.
		\end{equation}Suppose that $$\displaystyle{\sum_{i=1}^{\deg(f)} \left\vert g_\delta (\gamma _{\delta,i}) -g_0 (\gamma _{0,i}) \right\vert \leq R(\delta)};$$

	\end{enumerate}	
	
	\begin{enumerate}
		\item [(U2.2)] 
		$\esssup _\gamma \max_{i=1, \cdots, \deg(f)}	d_1(\gamma _{0,i},\gamma _{\delta,i}) \leq R(\delta),$ where the essential supremum is calculated with respect to $m$;
	\end{enumerate}

	\begin{enumerate}
		\item [(U2.3)] $G_0$ and $G_\delta$ are $R(\delta)$-close in the $\sup$ norm: for all $\delta$  $$d_2(G_{0}(x,y),G_{\delta }(x,y))\leq R(\delta) \ \forall (x,y) \in M\times K;$$ 
	\end{enumerate}
	
	\item [(U3)] For all $\delta \in (0,\delta_1)$, $f_\delta$ has an equilibrium state $m_{\delta}$, and $m_{\delta}$ is equivalent to $m$ for all $\delta \in [0,\delta_1)$. This means that $m\ll m_{\delta}$ and $m_{\delta} \ll m$ for all $\delta \in [0,\delta_1)$.  Moreover, 

\begin{equation}\label{nmbcvd}
	\func{J}:= \sup _ {\delta \in [0,\delta_1)} \left |\dfrac{dm_\delta}{dm} \right|_\infty < \infty, 
\end{equation}where $| \cdot |_ \infty$ norm is calculated with respect to $m$.
\end{enumerate}

\begin{remark}\label{toyiout}
	We observe that $g_\delta$ in Equation (\ref{iueyrtfd}) safisfies $g_\delta = \dfrac{1}{J_{f_\delta}}$, where $J_{f_\delta}(\gamma) := \dfrac{\lambda _ \delta h_ \delta\circ f_ \delta (\gamma)}{h_ \delta(\gamma)e^{\varphi(\gamma)}}$ is the Jacobian of $(f_\delta, m_\delta)$. Thus, condition (U2.1) concerns the distance between the Jacobians of $f_\delta$ and $f_0$. Moreover, by (U3), since $m \ll m_{\delta}$ for all $\delta$ and $m_{\delta}$ is $f_\delta$-invariant, it follows that $\sum_{i=1}^{\deg(f)}  g_\delta(\gamma _{\delta,i})=1$ $m$-almost everywhere. 
\end{remark}
\begin{enumerate}
	\item [(A1)] (Uniform Lasota-Yorke inequality) There exist constants $B_3>0$ and $0<\beta _3 <1$ such that for all $u \in H_\zeta$,  all $\delta \in [0,1)$,  and all $n \geq 1$, the following inequality holds:
		
	\begin{equation*}
			|\mathcal{\overline{L}}_{\varphi,\delta}^n(u)|_{\zeta} \leq B_3 \beta _3 ^n | u|_{\zeta} + B_3|u|_{\infty},
	\end{equation*}where $|u|_\zeta := H_\zeta (u) + |u|_{\infty}$ and $\mathcal{\overline{L}}_{\varphi,\delta}$ is the Ruelle-Perron-Frobenius operator of $f_{\delta }$ defined by applying Remark \ref{yturhfvb} to $f_\delta$.
\end{enumerate}

\begin{enumerate}
	\item [(A2)] For all $\delta \in [0,1), $ let $\alpha _\delta$,  $L_{1,\delta}$ and $|G_\delta|_ \zeta$ be the contraction rate $\alpha$ given by Equation (\ref{contracting1}) for $G_\delta$,  the constant $L_1$ given by (f1) for $f_\delta$, and the constant $|G|_ \zeta$  defined by Equation $(\ref{jdhfjdh})$, respectively.  Set  $ \beta_\delta:= (\alpha_\delta L_{1,\delta})^\zeta$ and $D_{2, \delta}:=\{\epsilon _{\rho, \delta} L_{1,\delta}^\zeta + |G_\delta|_ \zeta L_{1,\delta}^\zeta\}$. Suppose that, $$\sup _ \delta \beta_\delta <1$$and $$\sup _ \delta D_{2, \delta} < \infty.$$ 
\end{enumerate}

\begin{corollary}\label{nslfdflsdjlf}
	Let $\{F_{\delta }\}_{\delta \in [0,\delta_1)}$ an admissible $R(\delta)$-perturbation and $\gamma_{\delta, i}$ the $i$-th pre-image of $\gamma \in M$ by $f_\delta$, $i=1, \cdots, \deg(f_\delta)$. Then, for all positive measure $\mu \in \mathcal{H} ^\zeta $ which satisfy $\phi _1 = 1$ $m$-a.e., the following inequality holds: $$\left\vert \left\vert ({\func{F}_{0,\gamma _{0,i} }{_\ast }}- \func{F}_{0,\gamma _{\delta,i} }{_\ast })\mu |_{\gamma _{0,i}}\right\vert \right\vert _{W} \leq R(\delta)^\zeta( 2 \alpha^\zeta |\mu|_\zeta + |G|_\zeta ||\mu ||_\infty ) , \forall i=1, \cdots, \deg(f),$$where $F_{\delta,\gamma _{\delta,i}}$ is defined by equation (\ref{poier}), for all $\delta \in [0,1)$.
\end{corollary}
\begin{proof}
	To simplify the notation, we denote $F:=F_0$ and $\gamma:=\gamma_{0,i}$. Thus, we have
	
	\begin{eqnarray*}
		\left\vert \left\vert ({\func{F}_{0,\gamma _{0,i} }{_\ast }}- \func{F}_{0,\gamma _{\delta,i} }{_\ast })\mu |_{\gamma _{0,i}}\right\vert \right\vert _{W} &=& \left\vert \left\vert ({\func{F}_{\gamma }{_\ast }}- \func{F}_{\gamma _{\delta,i} }{_\ast })\mu |_{\gamma}\right\vert \right\vert _{W} \\&=& \left\vert \left\vert {\func{F}_{\gamma }{_\ast }\mu |_{\gamma}}- \func{F}_{\gamma _{\delta,i} }{_\ast }\mu |_{\gamma}\right\vert \right\vert _{W} \\&\leq& \left\vert \left\vert  {\func{F}_{\gamma }{_\ast }\mu |_{\gamma}}- \func{F}_{\gamma _{\delta,i} }{_\ast }\mu |_{\gamma_{\delta,i}}\right\vert\right\vert _{W} +  \left\vert \left\vert \func{F}_{\gamma _{\delta,i} }{_\ast }(\mu |_{\gamma_{\delta,i}} - \mu |_{\gamma} ) \right\vert\right\vert _{W}
	\end{eqnarray*}Since $\phi _1 = 1$ $m$-a.e., $\mu |_{\gamma_{\delta,i}} - \mu |_{\gamma}$ has zero average. Therefore, by lemmas \ref{opsdas} and \ref{apppoas}, by (U2.2) and Definition (\ref{Lips3}) applied on $\mu$, we get
	\begin{eqnarray*}
		\left\vert \left\vert ({\func{F}_{0,\gamma _{0,i} }{_\ast }}- \func{F}_{0,\gamma _{\delta,i} }{_\ast })\mu |_{\gamma _{0,i}}\right\vert \right\vert _{W} &\leq& \left\vert \left\vert  {\func{F}_{\gamma }{_\ast }\mu |_{\gamma}}- \func{F}_{\gamma _{\delta,i} }{_\ast }\mu |_{\gamma_{\delta,i}}\right\vert\right\vert _{W} +  \alpha ^\zeta \left\vert \left\vert \mu |_{\gamma_{\delta,i}} - \mu |_{\gamma}  \right\vert\right\vert _{W} \\&\leq& 	\alpha^\zeta |\mu|_\zeta  d_1(\gamma_{\delta,i}, \gamma)^\zeta  + |G|_\zeta d_1(\gamma_{\delta,i}, \gamma)^\zeta ||\mu ||_\infty \\&+&  \alpha ^\zeta | \mu |_\zeta d_1(\gamma_{\delta,i}, \gamma)^\zeta
		\\&\leq& R(\delta)^\zeta( 2 \alpha^\zeta |\mu|_\zeta + |G|_\zeta ||\mu ||_\infty ).
	\end{eqnarray*}

\end{proof}
\begin{lemma}\label{çhjghljk}
	Let $\{F_{\delta }\}_{\delta \in [0,\delta_1)}$ an admissible $R(\delta)$-perturbation and $\gamma_{\delta, i}$ the $i$-th pre-image of $\gamma \in M$ by $f_\delta$, $i=1, \cdots, \deg(f)$. Then, the following inequality holds: $$\left\vert \left\vert ({\func{F}_{0,\gamma _{\delta,i} }{_\ast }}- \func{F}_{\delta,\gamma _{\delta,i} }{_\ast })\mu |_{\gamma _{0,i}}\right\vert \right\vert _{W} \leq ||\mu|_{\gamma_{0,i}}|| R(\delta)^\zeta, \forall i=1, \cdots, \deg(f_\delta),$$ where $F_{\delta,\gamma _{\delta,i}}$ is defined by Equation (\ref{poier}), for all $\delta \in [0,1)$.
\end{lemma}

\begin{proof}
	To simplify the notation, we denote $\gamma:=\gamma_{\delta,i}$. Thus, by Definition (\ref{wasserstein}) and (U2.3), we have

	\begin{eqnarray*}
		\left\vert \left\vert ({\func{F}_{0,\gamma }{_\ast }}- \func{F}_{\delta,\gamma }{_\ast })\mu |_{\gamma _{0,i}}\right\vert \right\vert _{W} &=& \left\vert \left\vert ({\func{F}_{0, \gamma }{_\ast }}- \func{F}_{\delta, \gamma}{_\ast })\mu |_{\gamma _{0,i}}\right\vert \right\vert _{W} \\&=& \sup _{H_\zeta(u)\leq 1,|u|_{\infty }\leq 1} \left\vert \int {u}  d ({\func{F}_{0, \gamma }{_\ast }}\mu |_{\gamma _{0,i}} - \func{F}_{\delta, \gamma}{_\ast }\mu |_{\gamma _{0,i}}) \right\vert  \\&=& \sup _{H_\zeta(u)\leq 1,|u|_{\infty }\leq 1} \left\vert \int {u (G_0(\gamma,y))  - u (G_\delta(\gamma,y))}  d \mu |_{\gamma _{0,i}} \right\vert \\&\leq & \sup _{H_\zeta(u)\leq 1,|u|_{\infty }\leq 1}  \int { \left\vert u (G_0(\gamma,y))  - u (G_\delta(\gamma,y) ) \right\vert}  d \mu |_{\gamma _{0,i}}  \\&\leq&   \int {d_2(G_0(\gamma,y),G_\delta(\gamma,y))^\zeta}  d \mu |_{\gamma _{0,i}}  \\&\leq& R(\delta)^\zeta \left\vert \int {1}  d \mu |_{\gamma _{0,i}} \right\vert \\&\leq& R(\delta)^\zeta||\mu |_{\gamma _{0,i}}||_W.
	\end{eqnarray*}
	
\end{proof}

\begin{remark}\label{kjfgeiurytoi}
	By Corollary \ref{fff}, Definition \ref{normalized} and Remark \ref{jhdgflasçlçl} we have that 
	\begin{equation*}
		(\func {\overline{F}}_{\Phi}\mu)|_\gamma:=\dfrac{1}{\lambda h(\gamma)}\sum _{i=1}^{\deg(f)}{\func {F}_{\gamma_i*}\mu|_{\gamma_i}h(\gamma_i)e^{\varphi(\gamma_i)}}, 
	\end{equation*}for $m$-a.e. $\gamma \in M$ and all $\delta \in [0,\delta_1)$. For such a $\gamma$, defining $$g(\gamma):=\dfrac{1}{\lambda h\circ f}h(\gamma)e^{\varphi(\gamma)},$$we have that 
	\begin{equation}
		(\func {\overline{F}}_{\Phi}\mu)|_\gamma:=\sum _{i=1}^{\deg(f)}{\func {F}_{\gamma_i*}\mu|_{\gamma_i}g(\gamma_i)}. 
	\end{equation}This final expression will be useful from now on.
\end{remark}

In what follows, for a given admissible $R(\delta)$-perturbation $\{F_\delta \}_{\delta \in [0,\delta_1)}$, we denote by $\func {\overline{F}}_{\Phi, \delta}$ the operator from Definition \ref{normalized}, applied to the transformation $F_\delta$, instead of $F$. Then, by Remark \ref{kjfgeiurytoi} and for a given $\delta \in [0,\delta_1)$ we have 

\begin{equation}\label{good}
	(\func {\overline{F}}_{\Phi, \delta}\mu)|_\gamma:=\sum _{i=1}^{\deg(f)}{\func {F}_{\delta, \gamma_i*}\mu|_{\gamma_i}g_\delta(\gamma_i)},
\end{equation}for $m$-a.e. $\gamma \in M$ and all $\mu \in \mathbf{AB}_m$, where $$g_\delta(\gamma):=\dfrac{1}{\lambda h\circ f_\delta}h(\gamma)e^{\varphi(\gamma)}.$$

Understanding the action of the transfer operators $\{\func {\overline{F}}_{\Phi, \delta}\}_{[0,\delta_1)}$ associated with an admissible $R(\delta)$-perturbation on the space $\mathbf{L}^\infty (= \mathbf{L}^\infty_m)$ is essential, as the measure $m$, which defines $\mathbf{L}^\infty$, is fixed. The $F_\delta$-invariant probabilities $\mu_\delta$ project onto $m_\delta$, which is absolutely continuous with respect to a conformal measure $\nu_\delta$ with Jacobian $\lambda_\delta e^{-\varphi}$. Moreover, the measure $m_\delta = h_ \delta \nu_\delta$. The next result shows that a statement similar to Proposition \ref{kdjfjksdkfhjsdfk} still holds uniformly for $\func{\overline{F}}_{\Phi, \delta}$ for all $\delta$. Additionally, we recall that $\func{\overline{F}}_{\Phi, \delta}$ is defined by Definition \ref{normalized}, using the abbreviated notation introduced in Remark \ref{jhdgflasçlçl}.

\begin{proposition}\label{agorasim}
	The operator $\func {\overline{F}}_{\Phi,\delta}: \mathbf{L}^{\infty} \longrightarrow \mathbf{L}^{\infty}$ is a weak contraction for all $\delta \in [0,\delta_1)$. It holds, $||\func {\overline{F}}_{\Phi,\delta} \mu ||_\infty \leq ||\mu||_\infty$, for all $\mu \in \mathbf{L}^{\infty}$.
\end{proposition}
\begin{proof}
	
	The first thing to note is that $\overline{\mathcal{L}}_{\varphi,\delta} (1) = 1$ $m$-almost everywhere and for all $\delta \in [0,\delta_1)$. Indeed, since $m_\delta$ is $f_\delta$-invariant and its Jacobian is $J_{m_\delta,f_\delta} :=\dfrac{\lambda_\delta e^{-\varphi} h_ \delta \circ f_ \delta}{h_\delta }$ we have that (for $u \equiv 1$) 
	\begin{eqnarray*}
		\overline{\mathcal{L}}_{\varphi,\delta} (1) (\gamma)  &=& \frac{\mathcal{L}_{\varphi}(h_\delta)(\gamma)}{\lambda_\delta h_\delta(\gamma)} \\&=& \sum \dfrac{h_\delta (\gamma_{\delta,i}) e^{\varphi (\gamma_{\delta,i})}}{\lambda_ \delta h_\delta (\gamma)} \\&=& \sum \dfrac{1}{J_{m_\delta,f_\delta} (\gamma_{\delta,i})} \\&=&1
	\end{eqnarray*} for $m_\delta$-almost every $\gamma \in M$. Since, by U3, $m_\delta$ and $m$ are equivalents, we obtain $\overline{\mathcal{L}}_{\varphi,\delta} (1) = 1$ $m$-almost everywhere and for all $\delta \in [0,\delta_1)$ as we desired.

	Now let us complete the proof by applying equation (\ref{niceformulaac}), Corollary \ref{fff} and Remark \ref{yturhfvb} in the sequence of inequalities below: for $m$-almost every $\gamma \in \mathcal{F}^s (\approxeq M)$, we have
	
	\begin{eqnarray*}
		||(\func {\overline{F}}_{\Phi,\delta} \mu)|_\gamma ||_W 
		&\leq& \dfrac{1}{h_\delta}\sum _{i=1}^{\deg(f)}{||\dfrac{1}{\lambda_\delta} \func {F_{\delta,\gamma_{\delta,i}}}_{*}\mu|_{\gamma_{\delta, i}}h_\delta(\gamma _i)e^{\varphi(\gamma_{\delta,i})}||_o}
		\\ &\leq& \dfrac{1}{h_\delta \lambda_\delta} \sum _{i=1}^{\deg(f)}{ h_\delta(\gamma_{\delta,i}) e^{\varphi(\gamma_{\delta,i})}||\mu|_{\gamma_{\delta,i}}||_W}
		\\ &\leq& \dfrac{1}{h_\delta \lambda_\delta} ||\mu||_\infty \sum _{i=1}^{\deg(f)}{h_\delta (\gamma_i)e^{\varphi(\gamma_{\delta,i})}} 
		\\ &=& ||\mu||_\infty  \overline{\mathcal{L}}_{\varphi,h_\delta} (1)\\ &=& ||\mu||_\infty.
	\end{eqnarray*} Taking the essential supremum over $\gamma$, with respect to $m$, we have $$||\func {\overline{F}_{\Phi,\delta}} \mu ||_\infty \leq ||\mu||_\infty,$$which is the desired relation.
	
\end{proof}

\begin{lemma}\label{UF2ass}
	Suppose that $\{F_\delta \}_{\delta \in [0,1)}$ is an admissible $R(\delta)$-perturbation, and let $\func {\overline{F}}_{\Phi, \delta}$ and $%
	\mu_{\delta }$ denote the corresponding transfer operators and fixed points (i.e., the $F_\delta$ invariant probability measures in $\mathbf{S}^\infty$), respectively. If the family $\{\mu_{\delta }\}_{\delta \in [0,1)}$ satisfies 
	\begin{equation}\label{new1}
		|\mu_{\delta }|_\zeta \leq B_u,
	\end{equation}for all $\delta \in [0,\delta_1)$,
	then there exists a constant $C_{1}$ such that
	\begin{equation}\label{oiuteiyoey}
	||(\func {\overline{F}}_{\Phi, 0}-\func {\overline{F}}_{\Phi, \delta})\mu_{\delta }||_{{\infty}}\leq
	C_{1}R(\delta)^\zeta,
	\end{equation}
	for all $\delta \in [0,\delta_1)$, where $C_1:=|G_0|_\zeta + 3B_u +2$.
\end{lemma}

\begin{proof}

Our approach is to estimate the following quantity:
	
	\begin{equation}\label{12112}
		||(\func {\overline{F}}_{\Phi, 0}-\func {\overline{F}}_{\Phi, \delta})\mu_{\delta }||_{{\infty}}= \esssup_{M}{||(\func {\overline{F}}_{\Phi, 0}\mu_{\delta })|}_{\gamma }-(\func {\overline{F}}_{\Phi, 0})|_{\gamma}||_{W}.
	\end{equation}

	Let $f_{\delta,i}$ denote ($1\leq i\leq \deg(f)$) the branches of $f_{\delta}$ defined on the sets $P_{i} \in \mathcal{P}$ where  $\mathcal{P}$ depends on $\delta$. That is, each branch is given by $f_{\delta,i}=f_{\delta }|_{P_{i}}$. Moreover, recall that we denote  $\gamma_{\delta, i}:= f^{-1}_{\delta, i} (\gamma)$ for all $\gamma \in M$. Furthermore, by (U2.2), there exists $R(\delta)$ such that  
	
	\begin{equation}
		d_1(\gamma _{0,i},\gamma _{\delta,i}) \leq R(\delta) \ \ \forall i=1 \cdots \deg(f).
	\end{equation}Remember that, by (U1) $\deg (f_\delta) = \deg (f)$ for all $\delta \in [0,\delta_1)$.

Denoting $\func{F}_{\delta,\gamma_{\delta,i}}:=\func{F}_{\delta ,f_{\delta ,i}^{-1}(\gamma )}$ and setting $\mu := \mu_\delta$, we obtain from equation (\ref{good}) that for $\mu _{x}$-almost every $\gamma \in M$, the following holds:

	\[
	(\func {\overline{F}}_{\Phi, 0}\mu-\func {\overline{F}}_{\Phi, \delta}\mu )|_{\gamma }=\sum_{i=1}^{\deg(f)}%
	\func{F}_{0,\gamma _{0,i}}{_\ast}\mu |_{\gamma _{0,i}}g_0(\gamma _{0,i})%
	-\sum_{i=1}^{\deg(f)}{\func{F}_{\delta,\gamma _{\delta,i} }{_\ast }}\mu |_{\gamma _{\delta,i}}g_\delta (\gamma _{\delta,i}). 
	\]%
	Therefore, we obtain
	
	\begin{equation*}
		||(\func {\overline{F}}_{\Phi, 0}\mu-\func {\overline{F}}_{\Phi, \delta})\mu||_{\infty} \leq \func{A}    +   \func{B},
	\end{equation*}where 
	
	\begin{equation}\label{A}
		\func{A} := \esssup_{M}\left\vert \left\vert \sum_{i=1}^{\deg(f)}%
		{\func{F}_{0,\gamma _{0,i} }{_\ast }}\mu |_{\gamma _{0,i}}g_0(\gamma _{0,i})%
		-\sum_{i=1}^{\deg(f)}{\func{F}_{\delta,\gamma _{\delta,i} }{_\ast }}\mu |_{\gamma _{0,i}}g_\delta (\gamma _{\delta,i})\right\vert \right\vert _{W} 
	\end{equation}

	and 
	
	\begin{equation}	\label{B}
		\func{B} := \esssup_{M}\left\vert \left\vert \sum_{i=1}^{\deg(f)}%
		{\func{F}_{\delta,\gamma _{\delta,i} }{_\ast }}\mu |_{\gamma _{0,i}}g_{\delta}(\gamma _{\delta,i})%
		-\sum_{i=1}^{\deg(f)}%
	{\func{F}_{\delta,\gamma _{\delta,i} }{_\ast }}\mu |_{\gamma _{\delta,i}}g_{\delta}(\gamma _{\delta,i})|\right\vert \right\vert _{W}.
	\end{equation}
	
	We now proceed to bound	$\func {A}$ from equation (\ref{A}). By applying the triangle inequality in a similar manner, we obtain
	
	$$ \func{A} \leq \esssup_{M}\func{A}_1(\gamma) + \esssup_{M}\func{A}_2(\gamma),$$ where

	\begin{equation}
		\func{A}_1(\gamma) := \left\vert \left\vert \sum_{i=1}^{\deg(f)}%
		{\func{F}_{0,\gamma _{0,i} }{_\ast }}\mu |_{\gamma _{0,i}}g_{0}(\gamma _{0,i})
		-\sum_{i=1}^{\deg(f)}{\func{F}_{\delta,\gamma _{\delta,i} }{_\ast }}\mu |_{\gamma _{0,i}}g_{0}(\gamma _{0,i})\right\vert \right\vert _{W}
	\end{equation}and

	\begin{equation}
		\func{A}_2(\gamma) := \left\vert \left\vert \sum_{i=1}^{\deg(f)}%
		{\func{F}_{\delta,\gamma _{\delta,i} }{_\ast }}\mu |_{\gamma _{0,i}}g_{0}(\gamma _{0,i})
		-\sum_{i=1}^{\deg(f)}{\func{F}_{\delta,\gamma _{\delta,i} }{_\ast }}\mu |_{\gamma _{0,i}}g_{\delta}(\gamma _{\delta,i})\right\vert \right\vert _{W}.
	\end{equation}Each summand will be analyzed individually. For $\func{A}_1 $, we observe that

	\begin{eqnarray*}
		\func{A}_1(\gamma) &\leq &  \sum_{i=1}^{\deg(f)}%
		\left\vert \left\vert {\func{F}_{0,\gamma _{0,i} }{_\ast }}\mu |_{\gamma _{0,i}}g_{0}(\gamma _{0,i})
		-\sum_{i=1}^{\deg(f)}{\func{F}_{\delta,\gamma _{\delta,i} }{_\ast }}\mu |_{\gamma _{0,i}}g_{0}(\gamma _{0,i})\right\vert \right\vert _{W}
		\\&\leq &  \sum_{i=1}^{\deg(f)}%
		\left\vert \left\vert ({\func{F}_{0,\gamma _{0,i} }{_\ast }}- \func{F}_{\delta,\gamma _{\delta,i} }{_\ast })\mu |_{\gamma _{0,i}}\right\vert \right\vert _{W}g_{0}(\gamma _{0,i})
		\\&\leq &  \sum_{i=1}^{\deg(f)}%
		\left\vert \left\vert ({\func{F}_{0,\gamma _{0,i} }{_\ast }}- \func{F}_{0,\gamma _{\delta,i} }{_\ast })\mu |_{\gamma _{0,i}}\right\vert \right\vert _{W}g_{0}(\gamma _{0,i}) \\&+&  \sum_{i=1}^{\deg(f)}%
		\left\vert \left\vert ({\func{F}_{0,\gamma _{\delta,i} }{_\ast }}- \func{F}_{\delta,\gamma _{\delta,i} }{_\ast })\mu |_{\gamma _{0,i}}\right\vert \right\vert _{W}g_{0}(\gamma _{0,i}).
	\end{eqnarray*}For $\mu=\mu_{ \delta}$, since $\pi_{1*} \mu =m$, it follows that $\phi_1 \equiv 1$. Applying Remark \ref{toyiout}, Corollary \ref{nslfdflsdjlf}, and Lemma \ref{çhjghljk} to the above, we obtain
	
	\begin{eqnarray*}
		\func{A}_1(\gamma) &\leq &  \left(\sum_{i=1}^{\deg(f)}%
		g_{0}(\gamma _{0,i})\right) R(\delta)^\zeta( 2 \alpha^\zeta |\mu|_\zeta + |G|_\zeta ||\mu ||_\infty ) \\&+& \left(\sum_{i=1}^{\deg(f)}%
		g_{0}(\gamma _{0,i})\right)  R(\delta)^\zeta ||\mu |_{\gamma _{0,i}}||_W
		\\&\leq &    R(\delta)^\zeta (2B_u + |G_0|_\zeta +1). 
	\end{eqnarray*}
	
	For $\func{A}_2(\gamma)$, by (U2.1) we have
	\begin{eqnarray*}
		\func{A}_2(\gamma) &\leq&  \sum_{i=1}^{\deg(f)}%
		\left\vert \left\vert {\func{F}_{\delta,\gamma _{\delta,i} }{_\ast }}\mu |_{\gamma _{0,i}}g_{0}(\gamma _{0,i})
		-{\func{F}_{\delta,\gamma _{\delta,i} }{_\ast }}\mu |_{\gamma _{0,i}}g_{\delta}(\gamma _{\delta,i})\right\vert \right\vert _{W}
		\\&\leq& \sum_{i=1}^{\deg(f)}%
		\left\vert g_{0}(\gamma _{0,i})
		-g_{\delta}(\gamma _{\delta,i})\right\vert  \left\vert \left\vert {\func{F}_{\delta,\gamma _{\delta,i} }{_\ast }}\mu |_{\gamma _{0,i}}\right\vert \right \vert _{W}
		\\&\leq& \sum_{i=1}^{\deg(f)}%
		\left\vert g_{0}(\gamma _{0,i})
		-g_{\delta}(\gamma _{\delta,i})\right\vert
		\\&\leq&  R(\delta)^\zeta.
	\end{eqnarray*}We now estimate $\func{B}$. By Remark \ref{toyiout}, we observe that $\sum_{i=1}^{\deg(f)} \left\vert g_{\delta}(\gamma _{\delta,i})\right\vert =1$ $m$-almost every point. Consequently, we obtain
		\begin{eqnarray*}
		\func{B} &\leq&  \esssup_{M} \sum_{i=1}^{\deg(f)}%
		\left\vert \left\vert {\func{F}_{\delta,\gamma _{\delta,i} }{_\ast }}\mu |_{\gamma _{0,i}}g_{\delta}(\gamma _{\delta,i})
		-{\func{F}_{\delta,\gamma _{\delta,i} }{_\ast }}\mu |_{\gamma _{\delta,i}}g_{\delta}(\gamma _{\delta,i})\right\vert \right\vert _{W}
		\\&\leq&  \esssup_{M} \sum_{i=1}^{\deg(f)} \left\vert g_{\delta}(\gamma _{\delta,i})\right\vert
		\left\vert \left\vert {\func{F}_{\delta,\gamma _{\delta,i} }{_\ast }}(\mu |_{\gamma _{0,i}}-\mu |_{\gamma _{\delta,i}})\right\vert \right\vert _{W}
		\\&\leq&  \esssup_{M} \sum_{i=1}^{\deg(f)} \left\vert g_{\delta}(\gamma _{\delta,i})\right\vert
		\left\vert \left\vert \mu |_{\gamma _{0,i}}-\mu |_{\gamma _{\delta,i}}\right\vert \right\vert _{W}
		\\&\leq&  \esssup_{M} \sum_{i=1}^{\deg(f)} \left\vert g_{\delta}(\gamma _{\delta,i})\right\vert
		d_1(\gamma _{\delta,i},\gamma _{0,i})^\zeta|\mu|_\zeta
		\\&\leq& \esssup_{M} \sum_{i=1}^{} \left\vert g_{\delta}(\gamma _{\delta,i})\right\vert
		R(\delta) ^\zeta |\mu|_\zeta
		\\&\leq&  R(\delta) ^\zeta  B_u.
	\end{eqnarray*}Since $\zeta <1$,  it follows that $\delta \leq \delta ^\zeta$. Therefore, combining these observations, we obtain
	\begin{eqnarray*}
		||(\func{F_0{_\ast }}-\func{F_\delta{_\ast }})\mu_{\delta }||_{\infty} & \leq & \func{A}  +   \func{B}
		\\& \leq & \esssup_{M} \func{A}_1(\gamma)    + \esssup_{M} \func{A}_2(\gamma) +   \func{B}
		\\& \leq & R(\delta)^\zeta (2B_u + |G_0|_\zeta +1) +  R(\delta)^\zeta + R(\delta) ^\zeta  B_u
		\\& \leq &C_1 R(\delta) ^\zeta,
	\end{eqnarray*}where $C_1:=|G_0|_\zeta + 3B_u +2.$
\end{proof}

The following result is a key tool in proving Theorem \ref{htyttigu}. It establishes that the function $$\delta \longmapsto |\mu_{\delta }|_\zeta$$(see Definition \ref{Lips3}) is uniformly bounded, where  $\{\mu_\delta\}_{\delta \in [0,1)}$ denotes the family of $F_{\delta }$-invariant probability measures associated with an admissible perturbation $\{F_{\delta }\}_{\delta \in [0,1)}$ of $F(=F_0)$. 

We begin with a preliminary lemma.

\begin{lemma}	Given an admissible $R(\delta)$-perturbation $\{F_{\delta }\}_{\delta \in [0,1)}$, there exist constants $0<\beta_u<1$ and $D_{2,u}>0$ such that for every measure $\mu \in \mathcal{H} _\zeta^{+}$ satisfying $\phi _1 = 1$ $m$-almost everywhere, the following inequality holds:
	\begin{equation}\label{er}
		|\Gamma_{\func{\overline{F}_{\Phi,\delta} ^n \mu}}^\omega|_{\zeta}  \leq \beta_u ^n |\Gamma _\mu^\omega|_\zeta + \dfrac{D_{2,u}}{1-\beta_u}||\mu||_\infty,
	\end{equation}for all $\delta \in [0,1)$ and all $n \geq 0$.
	\label{las123rtryrdfd2}
\end{lemma}
\begin{proof}
By applying Corollary \ref{kjdfhkkhfdjfh} to each mapping $F_\delta$, we obtain
	\begin{equation*}
		|\Gamma_{\func{\overline{F}_{\Phi,\delta} \mu}}^\omega|_\zeta \leq  \beta_\delta |\Gamma_{\mu}^\omega |_\zeta + \frac{D_{2,\delta}}{1-\beta_\delta} ||\mu||_{\infty}, \ \forall \delta \in [0,1),  \label{lasotaingt234dffggdgh2}
	\end{equation*}where $\beta_\delta:= (\alpha_\delta L_{1,\delta})^\zeta$ and $$D_{2,\delta}:=\{\epsilon _{\rho,\delta}  L_{1,\delta}^\zeta + |G_\delta|_ \zeta L_{1,\delta}^\zeta\}.$$Using assumption (A2), we then define $$\beta_u:= \displaystyle{\sup _ \delta \beta_\delta} \ \text{and} \ D_{2,u}:= \displaystyle{\sup_\delta D_{2,\delta}},$$ which completes the proof.
	
\end{proof}	
\begin{lemma}
	\label{thshgf}
	Consider an admissible $R(\delta)$-perturbation $\{F_{\delta }\}_{\delta \in [0,1)}$ and let $\mu_{\delta }$ be the family of $F_\delta$-invariant probability measures in $\mathbf{S}^\infty$ (which is unique for each $\delta$), for all $\delta \in[0,1)$. Then, there exists $B_u>0$ such that 
	\begin{equation*}
		|\mu_{\delta }|_\zeta\leq B_u,
	\end{equation*}for all $\delta \in[0,1)$. In particular, $\{\mu_{\delta }\}_{\delta \in [0,1)}$ satisfies Equation (\ref{new1}) and all admissible $R(\delta)$-perturbation $\{F_{\delta }\}_{\delta \in [0,1)}$ satisfies Equation (\ref{oiuteiyoey}).
\end{lemma}
\begin{proof}
According to Theorem \ref{belongsss}, for each $\delta \in [0,1)$ we denote by $\mu_\delta \in \mathbf{S}^\infty$ the unique $F_\delta$-invariant probability measure in $\mathbf{S}^\infty$. For a fixed $\delta \in [0,1)$and following Definition \ref{defd}, we consider a path $$\Gamma^\omega _{\mu _{\delta}}|_{\widehat{M_\delta}}$$which belongs to the equivalence class $\Gamma _{\mu _{\delta}}$. Additionally, we consider the measure $m_1$ defined in Remark \ref{riirorpdf} and its iterates $\func{\overline{F}^n_{\Phi, \delta}} m_1$. Applying Remark \ref{riirorpdf} to $\func{\overline{F}^n_{\Phi, \delta}} m_1$ yields the sequence of representations $$\{ \Gamma_{\func{\overline{F}^n_{\Phi, \delta}} m_1} ^{\omega_n} \}_{n \in \mathbb{N}},$$ where each function $\func{\overline{F}^n_{\Phi, \delta}} m_1$ belongs to the equivalence class corresponding to $\func{\overline{F}^n_{\Phi, \delta}} m_1$. For notational simplicity, we define $$\Gamma_{\delta}^n:=\Gamma_{\func{\overline{F}^n_{\Phi, \delta}} m_1} ^{\omega_n}|_{\widehat{M_\delta}} \ \text{and} \ \Gamma_\delta:=\Gamma^\omega _{\mu _{\delta}}|_{\widehat{M_\delta}}.$$

By Theorem \ref{quasiquasiquasi}, the sequence $\{ \func{\overline{F}^n_{\Phi, \delta}} m_1\}_{n \in \mathbb{N}}$ converges to $\mu_\delta$ in $\mathbf{L}^{\infty}$, which implies that the sequence $\{ \Gamma_{\delta}^n \}_{n \in \mathbb{N}}$ converges $m$-a.e. to $\Gamma_\delta \in \Gamma_{\mu_\delta}$ in $\mathcal{SB}(K)$ with respect to the metric defined in Definition \ref{wasserstein}. In particular, $\{ \Gamma_{\delta}^n \}_{n \in \mathbb{N}}$ converges pointwise to $\Gamma_{\delta}$ on a full-measure subset $\widehat{M_\delta}\subset M$.

Therefore, as we shall see, it holds that $|\Gamma_{n,\delta}|_\zeta \longrightarrow |\Gamma _\delta|_\zeta$ as $n \rightarrow \infty$. Indeed, for any $x,y \in \widehat{M_\delta}$
	
	\begin{eqnarray*}
		\lim _{n \longrightarrow \infty} {\dfrac{||\Gamma_{\delta}^n (x) - \Gamma _{ \delta}^n(y)||_W}{d_1(x,y)^\zeta}} &= & \dfrac{||\Gamma _\delta (x) - \Gamma _\delta (y)||_W}{d_1(x,y)^\zeta}.
	\end{eqnarray*}Moreover, by Lemma \ref{thshgf}, the left-hand side is uniformly bounded by $|\Gamma_{n, \delta}|_\zeta \leq  \dfrac{D_u}{1-\beta_u}$  for all $n\geq 1$. Hence, 
	\begin{eqnarray*}
		\dfrac{||\Gamma _\delta (x) - \Gamma _\delta (y)||_W}{d_1(x,y)^\zeta}&\leq & \dfrac{D_{u}}{1-\beta _u}
	\end{eqnarray*} which implies that $$|\Gamma^\omega_{\mu _{\delta}}|_\zeta \leq\dfrac{D_{u}}{1-\beta _u}.$$Taking the infimum over all representations yields $|\mu _{\delta}|_\zeta \leq \dfrac{D_{u}}{1-\beta _u}$. This completes the proof of the first part with the definition $B_u:=\dfrac{D_{u}}{1-\beta _u}$. The second part follows directly from Lemma \ref{UF2ass}.
\end{proof}

\subsection{Perturbation of Operators}\label{perturbationoperators}

\begin{definition}
	Let $(B_{w}, ||\cdot||_{w} )$ and $(B_{s}, ||\cdot||_{s} )$ be vector spaces, which satisfies $B_{s} \subset B_{w}$ and $||\cdot||_{s}\geq ||\cdot||_{w}$. Suppose that $\func{T}  _{\delta }: B_{w} \longrightarrow B_{w}$ and $\func{T}  _{\delta }: B_{s} \longrightarrow B_{s}$ are well defined and, for all $\delta \in [0,1)$, $\mu_{\delta }\in B_{s}$ denotes a fixed point for $\func{T}  _{\delta }$. Moreover, assume that the following holds: 
	\begin{enumerate}
		\item [(O1)] There are $C\in \mathbb{R}^+$ and a function $\delta \longmapsto R(\delta) \in \mathbb{R}^+$ such that $$\lim_{\delta \rightarrow 0^+} {R(\delta)\log (\delta)}=0$$ and%
		\begin{equation*}
			||(\func{T}  _{0}-\func{T}  _{\delta })\mu_{\delta }||_{w}\leq R(\delta) ^\zeta C \ \forall \delta \in [0,1); 
		\end{equation*}
		
		\item [(O2)] There is $\func{Y}>0$ such that it holds $$||\mu_{\delta}||_{s}\leq \func{Y},  \ \forall \delta \in [0,1);$$
		\item[(O3)] There exists $0<\rho_2 <1$ and $C_2>0$ such that  it holds $$||\func{T}^{n}_0 \mu||_{w} \leq \rho _2 ^n C_2 ||\mu||_{s} \ \forall \ \mu \in \mathbf{V}_s: =\{\mu \in B_{s}: \mu(\Sigma)=0 \};$$
		\item[(O4)] There exists $M_2 >0$ such that for all $\delta \in [0,1)$, all $n \in \mathbb{N}$, and all $\nu \in B_{s}$, it holds $||\func{T} _{\delta
		} ^{ n}\nu||_{w}\leq M_{2}||\nu||_{w}.$
	\end{enumerate}A family of operators that satisfies (O1), (O2), (O3) and (O4) is called a \textbf{$(R(\delta), \zeta)$-family of operators}. We also refer to $(B_{w}, ||\cdot||_{w} )$ and $(B_{s}, ||\cdot||_{s} )$ as the \textbf{weak} and \textbf{strong} spaces of the family, respectively.
\end{definition}

Lemma \ref{dlogd} establishes a general, quantitative relationship between the variation of the fixed points $\{\mu_ \delta\}_{\delta \in [0,1)}$ of a $(R(\delta), \zeta)$-family of operators and the parameter $\delta$. In particular, it shows that the mapping $$\delta \mapsto \mu_{\delta} \ \text{with} \ \delta \longmapsto \func{T}_{\delta } \longmapsto \mu_{\delta }, \ \ \delta \in [0,1)$$ is continuous at $0$ with respect to the norm $||\cdot||_w$ and provides an explicit bound for its modulus of continuity, namely, $D_1 R(\delta)^\zeta \log \delta $, where $D_1\leq 0$. The proof is omitted here; the reader is referred to \cite{RRRSTAB} for details.

\begin{lemma} [Quantitative stability for fixed points of operators]	\label{dlogd}
	Suppose $\{\func{T}_{\delta }\}_{\delta \in \left[0, 1 \right)}$ is a $(R(\delta),\zeta)$-family of operators, where $\mu_{0}$ is the unique fixed point of $\func{T}_{0}$ in $B_{w}$ and $%
	\mu_{\delta }$ is a fixed point of $\func{T} _{\delta }$. Then, there exist constants $D_1 < 0$ and $\delta _0 \in (0,1)$ such that for all $\delta \in [0,\delta _0)$, it holds 
	
	\begin{equation*}
		||\mu_{\delta }-\mu_{0}||_{w}\leq D_1 R(\delta)^\zeta \log \delta.
	\end{equation*}
\end{lemma}

\begin{lemma}\label{rrr}
	Let $\{F_\delta\}_{\delta \in [0,1)}$ be an admissible $R(\delta)$-perturbation and let $\{\overline{\func{F}}_ {\Phi, \delta}\}_{\delta \in [0,1)}$ be the induced family of transfer operators. Then, $\{\overline{\func{F}}_ {\Phi, \delta}\}_{\delta \in [0,1)}$ is an $(R(\delta),\zeta)$-family of operators with $(\mathbf{L}^{\infty}, || \cdot ||_\infty)$ and  $(\mathbf{S}^\infty,||\cdot||_{\mathbf{S}^\infty})$ as the weak and strong spaces of the family, respectively.
\end{lemma}

\begin{proof}

	We need to show that $\{\overline{\func{F}}_ {\Phi, \delta}\}_{\delta \in [0,1)}$ satisfies properties O1, O2, O3, and O4. To prove O2, note that by Proposition \ref{ttty}, Proposition \ref{kdjfjksdkfhjsdfk}, Lemma \ref{thshgf}, condition (A1), Equation (\ref{nmbcvd}) from U3, and the fact that $||\mu_\delta||_\infty = |\phi _{1,\delta}|_{\infty}$ since $\mu_\delta$ is a probalbility measure, we obtain 
	\begin{eqnarray*}
		||\func{\overline{F} _{\Phi, \delta} ^n}  \mu_{\delta }||_{\mathbf{S}^\infty} &=& |\mathcal{\overline{L}}_{\varphi,\delta}^n( \phi _{1,\delta}) |_{\zeta} + ||\func{\overline{F} _{\Phi, \delta} ^n}  \mu_{\delta }||_{\infty}  \\&\leq & B_3 \beta _3 ^n |\phi _{1,\delta}|_{\zeta} + B_3|\phi _{1,\delta}|_\infty + || \mu_\delta||_{\infty} \\&\leq & B_3 \beta _3 ^n |\mu _\delta|_{\zeta} + B_3 \func{J} + |\phi _{1,\delta}|_\infty \\&\leq & B_3 B_u + B_3 \func{J} + \func{J}.
	\end{eqnarray*}Thus, if $\mu_{\delta }$ is a fixed probability measure for the operator  $\func{\overline{F} _{\Phi, \delta}}$, the inequality above implies that O2 holds with $\func{Y}=B_3 B_u + B_3 \func{J} + \func{J}$.
	
	A direct application of Lemma \ref{UF2ass} and Lemma \ref{thshgf} establishes O1. Finally, properties O3 and O4 follow from Theorem \ref{5.8} and Proposition \ref{agorasim}, respectively, applied to each $F_\delta$.
\end{proof}

\begin{athm}[Quantitative stability for deterministic perturbations]
	Let $\{F_{\delta }\}_{\delta \in [0,1)}$ be an admissible $R(\delta)$-perturbation. Let $\mu_0$ be the unique $F$-invariant probability in $\mathbf{S}^\infty$. Denote by $\mu_\delta$ the invariant measure of $F_\delta$ in $\mathbf{S}^\infty$, for all $\delta$. Then, there exist constants $D_2 < 0$ and $\delta _1 \in (0,\delta_0)$ such that for all $\delta \in [0,\delta _1)$, it holds
	\begin{equation}\label{stabll}
		||\mu_{\delta }-\mu_{0}||_{\infty}\leq D_2R(\delta)^\zeta \log \delta .
	\end{equation}
	\label{d}
\end{athm}

\begin{proof}
We directly apply the above results along with Theorem \ref{dlogd}, thereby completing the proof of Theorem \ref{htyttigu}. 
\end{proof}

\begin{remark}
	A straightforward computation yields $||\cdot ||_W \leq ||\cdot||_\infty$. Then, supposing that $\{F_\delta\}_{\delta \in [0,1)}$ satisfies Theorem (\ref{htyttigu}), it holds $$||\mu_{\delta }-\mu_{0}||_{W}\leq AR(\delta)^\zeta \log \delta ,$$for some $A>0$. Therefore, for all $\zeta$-Holder function $g:\Sigma \longrightarrow \mathbb{R}$, the following estimate holds $$\left|\int{g}d\mu_\delta - \int{g}d\mu_0\right| \leq A |g|_{\zeta} R(\delta) ^\zeta \log \delta,$$where $|g|_{\zeta} = |g|_\infty + H_\zeta(g)$ (see equation (\ref{lipsc}), for the definition of $H_\zeta(g)$). Thus, for all $\zeta$-Holder function, $g:\Sigma \longrightarrow\mathbb{R}$, the limit $\displaystyle{\lim _{\delta \longrightarrow 0} {\int{g}d\mu_\delta} = \int{g}d\mu_0}$ holds, with a rate of convergence smaller than or equal to $R(\delta)^\zeta \log \delta$.
\end{remark}

Many interesting perturbations of $F$ guarantee the existence of a linear $R(\delta)$. For instance, perturbations within topologies defined on the set of skew-products, induced by the $C^r$ topologies, satisfy this condition. In particular, if the function $R(\delta)$ is of the form $$R(\delta)=K_5 \delta,$$ for all $\delta$, where $K_5$ is a constant, we immediately obtain the following corollary.  

\begin{corollary}[Quantitative stability for deterministic perturbations with a linear $R(\delta)$]
	Let $\{F_{\delta }\}_{\delta \in [0,1)} $ be an admissible $R(\delta)$-perturbation, where $R(\delta)$ is defined by $R(\delta)=K_5\delta$. Denote by $\mu_\delta$ the unique invariant probability of $F_\delta$ in $\mathbf{S}^\infty$, for all $\delta$. Then, there exist constants $D_2 <0$ and $\delta _1 \in (0,\delta_0)$ such that for all $\delta \in [0,\delta _1)$, it holds
	\begin{equation*}
		||\mu_{\delta }-\mu_{0}||_{\infty}\leq D_2\delta^\zeta \log \delta.
	\end{equation*}
	\label{htyttigui}
\end{corollary}

\section{Extending the Class of Potentials}\label{ex}

A natural question is whether the results obtained so far extend to a broader class of potentials than $\mathscr{P}_\Sigma$ (see Equation (\ref{PPP})), in particular to potentials that are not necessarily constant along the fibres. To address this question, and to isolate it as a co-homological problem, we introduce the class $\mathcal{S}$ of dynamical systems $F$ satisfying the following properties:
\begin{enumerate}
	\item $F$ satisfies assumptions (f1), (f2), and (f3), and uniformly contracts all vertical fibres;
	\item there exists a class of potentials $\overline{\mathscr{P}}_\Sigma$ such that, for every
	$\overline{\Phi} \in \overline{\mathscr{P}}_\Sigma$, one can associate a potential
	$\Phi \in \mathscr{P}_\Sigma$ for which the pairs $(F,\Phi)$ and $(F,\overline{\Phi})$
	share the same equilibrium states.
\end{enumerate}Note that, in order to belong to the class $\mathcal{S}$, the map $G$ must satisfy a condition stronger than~(H1).

In this case, all results established in the previous sections extend naturally to potentials in 
$\overline{\mathscr{P}}_\Sigma$. As we shall see, this phenomenon indeed occurs for an interesting class 
of dynamical systems and for an equally interesting class of potentials $\overline{\mathscr{P}}_\Sigma$.

As an immediate consequence, we obtain the following results. Note that, by the definition of $\mathcal{S}$ and Theorem~\ref{belongsss}, in Theorems E--H the probability measure $\mu_0$ is an equilibrium state for the potential $\Phi \in \mathscr{P}_\Sigma$.

\begin{athm}\label{e}
	Suppose that $F\in \mathcal{S}$, $F$ satisfies (f1), (f2), (f3), (H2), $(\alpha \cdot L)^\zeta<1$ and $\Phi \in \overline{\mathscr{P}}_\Sigma$. Let $\mu_0$ be the unique equilibrium state of $\Phi$ in $\mathbf{S}^\infty$. There exists a constant $0 < \tau_2 < 1$ such that, for every constant fiber function 
	$\psi: \Sigma \to \mathbb{R}$ satisfying $\psi(\cdot, y) \in L^1(m)$ for all $y$, and for every 
	$\varphi \in \ho_\zeta(\Sigma)$, we have
	\[
	\left| 
	\int (\psi \circ F^n)\, \varphi \, d\mu_0 
	- \int \psi \, d\mu_0 \int \varphi \, d\mu_0 
	\right| 
	\le \tau_2^n D(\psi, \varphi) \quad \forall\, n \ge 1,
	\]
	where $D(\psi, \varphi) > 0$ is a constant depending on $\psi$ and $\varphi$.
\end{athm}

\begin{athm}\label{f}
		Suppose that $F\in \mathcal{S}$, $F$ satisfies (f1), (f2), (f3), (H2), $(\alpha \cdot L)^\zeta<1$ and $\Phi \in \overline{\mathscr{P}}_\Sigma$. Let $\mu_0$ be the unique equilibrium state of $\Phi$ in $\mathbf{S}^\infty$. There exists a constant $0 < \tau _3 < 1$ such that, for all $\psi \in L^1_{\mu_0}(\mathcal{F}_0)$ and all $\varphi \in \ho_\zeta(\Sigma)$, it holds
	\[
	\left| 
	\int (\psi \circ F^n)\, \varphi \, d\mu_0 
	- \int \psi \, d\mu_0 \int \varphi \, d\mu_0 
	\right| 
	\le \tau_3^n D(\psi, \varphi) \quad \forall\, n \ge 1,
	\]
	where $D(\psi, \varphi) > 0$ is a constant depending on $\psi$ and $\varphi$.
\end{athm}

\begin{athm}[\textbf{Central Limit Theorem}]\label{g}
		Suppose that $F\in \mathcal{S}$, $F$ satisfies (f1), (f2), (f3), (H2), $(\alpha \cdot L)^\zeta<1$ and $\Phi \in \overline{\mathscr{P}}_\Sigma$. Let $\mu_0$ be the unique equilibrium state of $\Phi$ in $\mathbf{S}^\infty$. 
	Given a H\"older continuous function $\varphi$, define
	\[
	\sigma_\varphi^2 := \int \phi^2 \, d\mu + 2 \sum_{j=1}^{\infty} \int \phi \, (\phi \circ F^j) \, d\mu,
	\quad \text{where } \phi = \varphi - \int \varphi \, d\mu.
	\]
	Then $\sigma_\varphi < \infty$ and $\sigma_\varphi = 0$ if and only if 
	$\varphi = u \circ F - u$ for some $u \in L^2(\mu)$. 
	Moreover, if $\sigma_\varphi > 0$, then for every interval $A \subset \mathbb{R}$,
	\[
	\lim_{n \to \infty}
	\mu\left(x \in M : 
	\frac{1}{\sqrt{n}} \sum_{j=0}^{n-1}
	\left( \varphi(F^j(x)) - \int \varphi \, d\mu \right)
	\in A
	\right)
	= \frac{1}{\sigma_\varphi \sqrt{2\pi}} \int_A e^{-\frac{t^2}{2\sigma_\varphi^2}} \, dt.
	\]
\end{athm}

\begin{athm}[Quantitative stability for deterministic perturbations]\label{h}
	Suppose that $F\in \mathcal{S}$, $F$ satisfies (f1), (f2), (f3), (H2), $(\alpha \cdot L)^\zeta<1$ and $\Phi \in \overline{\mathscr{P}}_\Sigma$. Let $\mu_0$ be the unique equilibrium state of $\Phi$ in $\mathbf{S}^\infty$. Let $\{F_{\delta }\}_{\delta \in [0,1)}$ be an admissible $R(\delta)$-perturbation.  Denote by $\mu_\delta$ the invariant measure of $F_\delta$ in $\mathbf{S}^\infty$, for all $\delta$. Then, there exist constants $D_2 < 0$ and $\delta _1 \in (0,\delta_0)$ such that for all $\delta \in [0,\delta _1)$, it holds
	\begin{equation}\label{stabll}
		||\mu_{\delta }-\mu_{0}||_{\infty}\leq D_2R(\delta)^\zeta \log \delta .
	\end{equation}
	\label{htyttigu}
\end{athm}

\subsection{An example of $\mathcal{S}$ and $\overline{\mathscr{P}}_\Sigma$}\label{SS}

In what follows, $\ho_\zeta(\Sigma)$ denotes the space of real-valued $\zeta$-H\"older continuous functions defined on $\Sigma := M \times K$. Throughout this discussion, we assume $F$ to be continuous.

For any $\bar{\Phi} \in \ho_\zeta(\Sigma)$ and a fixed element $y_0 \in K$ (as specified in Equation~\eqref{fixedfiber}), we define the function $\bar{\Phi}_{y_0} : M \to \mathbb{R}$ by $\bar{\Phi}_{y_0}(x) := \bar{\Phi}(x, y_0)$. We then introduce the family $\mathcal{S}$ of maps as follows:
\begin{equation}\label{fixedfiber}
    \mathcal{S} := \{ F : \Sigma \to \Sigma \mid F \text{ is continuous and } \exists \, y_0 \in K \text{ s.t. } G(x, y_0) = y_0, \, \forall \, x \in M \}. 
\end{equation}
Furthermore, we define the set of potentials $\overline{\mathscr{P}}_\Sigma \subset \ho_\zeta(\Sigma)$ by 
\begin{equation}
    \overline{\mathscr{P}}_\Sigma := \{ \bar{\Phi} \in \ho_\zeta(\Sigma) \mid \bar{\Phi}_{y_0} \in \mathscr{P}_M \}.
\end{equation}

\begin{remark}
	If the fibre $K$ can be decomposed as a finitely union $K=K_1\cup \cdots \cup K_n$ of pairwise disjoint compact sets $K_1, \cdots, K_n$ then the condition $G(x, y_0)=y_0$ for all $x\in M$ in the above definition can be replaced by $G_i(x, y_i)=y_i$ for all $x\in M$ and some $y_i\in K_i$, $i=1, \cdots, n$.
	In fact, since $M\times K$ is a product space and $K=K_1\cup \cdots \cup K_n$ we may define $n$ fibre dynamics $G_i:M\times K_i\to K_i$ by $G_i(x, y)=G(x, y)$ when $y\in K_i$, for each $i=1,\cdots, n$. 
\end{remark}

We say that two potentials $\bar{\Phi}, \Phi : M \times K \to \mathbb{R}$ are
\emph{co-homologous} if there exists a continuous function
$u : M \times K \to \mathbb{R}$ such that
\[
\Phi = \bar{\Phi} - u + u \circ F.
\]

\begin{proposition}\label{homologo}
Let $\bar{\Phi} : M \times K \to \mathbb{R}$ be a H\"older continuous potential.
Then there exists a H\"older continuous potential
$\Phi : M \times K \to \mathbb{R}$, independent of the stable direction, such
that:
\begin{enumerate}
	\item $\Phi$ is co-homologous to $\bar{\Phi}$;
	\item $P(F,\bar{\Phi}) = P(F,\Phi)$;
	\item $(F,\Phi)$ and $(F,\bar{\Phi})$ have the same equilibrium states.
\end{enumerate}
\end{proposition}

By Equation~(8) in \cite{RA3}, if $F$ is continuous, satisfies
\eqref{fixedfiber} for some $y_0 \in K$, and
$\bar{\Phi} : M \times K \to \mathbb{R}$ is H\"older continuous, then the
fibre-constant potential $\Phi$ defined by
\[
\Phi(x,y) := \bar{\Phi}(x,y_0), \qquad \forall (x,y) \in M \times K,
\]
is co-homologous to $\bar{\Phi}$. This is precisely the potential whose existence
is guaranteed by Proposition~\ref{homologo}. In particular, the pairs
$(F,\Phi)$ and $(F,\bar{\Phi})$ have the same equilibrium states.

Moreover, $\Phi$ can be written as
\[
\Phi = \bar{\varphi}_{y_0} \circ \pi_1,
\]
where $\bar{\varphi}_{y_0} : M \to \mathbb{R}$ is defined by
\[
\bar{\varphi}_{y_0}(x) := \bar{\Phi}(x,y_0), \qquad \forall x \in M.
\]Therefore, provided that $\bar{\varphi}_{y_0} \in \mathscr{P}_M$, the class $\mathcal{S}$ defined above provides a family of systems to which all the results of this article can be applied, specifically for the set of potentials $\overline{\mathscr{P}}_\Sigma$ that are not constant along fibers.

\begin{remark}
	Other classes $\mathcal{S}$ and $\overline{\mathscr{P}}_\Sigma$ can also be constructed. 
	For instance, see Proposition~7.1 in~\cite{GV}, where the authors obtain a result analogous to 
	Proposition~\ref{homologo} for homeomorphisms.
\end{remark}

\end{document}